\documentclass[oneside]{amsart}

\usepackage{subfiles}	
\usepackage[utf8]{inputenc}
\usepackage[english]{babel}
\usepackage{amsmath,amsthm}
\usepackage{amsfonts}
\usepackage{amssymb}
\usepackage{graphicx}
\usepackage[]{geometry}
\usepackage{tikz}
\usepackage{tikz-3dplot,pgfplots}
\usepackage{color}
\usepackage{enumerate}
\usepackage[all]{xy}
\usepackage[normalem]{ulem}

\theoremstyle{definition}
\newtheorem{definition}{Definition}[section]
\newtheorem{remark}[definition]{Remark} 
\newtheorem{defn}[definition]{Definition} 
\newtheorem{exam}[definition]{Example} 
\newtheorem*{question}{Question} 

\usepackage{cleveref}

\theoremstyle{plain}
\newtheorem{theorem}[definition]{Theorem}

\newtheorem{corollary}[definition]{Corollary}
\newtheorem{lemma}[definition]{Lemma}
\newtheorem{proposition}[definition]{Proposition}

\newtheorem{innercustomthm}{Theorem} 
\newenvironment{theorem*}[1]
{\renewcommand\theinnercustomthm{#1}\innercustomthm}
{\endinnercustomthm}

\newtheorem{innercustomprop}{Proposition} 
\newenvironment{prop*}[1]
{\renewcommand\theinnercustomprop{#1}\innercustomprop}
{\endinnercustomprop}

\newtheorem{innercustomcor}{Corollary} 
\newenvironment{corollary*}[1]
{\renewcommand\theinnercustomcor{#1}\innercustomcor}
{\endinnercustomthm}

\newcommand{\Z}{\ensuremath{{\mathbb{Z}}}}
\newcommand{\R}{\ensuremath{{\mathbb{R}}}}
\newcommand{\C}{\ensuremath{{\mathbb{C}}}}

\newcommand{\N}{\ensuremath{{\mathbb{N}}}}
\newcommand{\G}{\ensuremath{{\Gamma}}}
\newcommand{\AG}{A_{\Gamma}}
\newcommand{\WG}{\ensuremath{{W_{\Gamma}}}}
\newcommand{\DG}{\mathcal{D}_\Gamma} 
\newcommand{\DGP}{\ensuremath{{\mathcal{D}_\Gamma^+}}}

\newcommand{\SF}{\ensuremath{{\mathcal{S}^f}} } 

\newcommand{\lcm}{\operatorname{lcm}}
\newcommand{\llm}{\operatorname{lcm}_L}

\newcommand{\calG}{\ensuremath{\mathcal{G}}}
\newcommand{\calM}{\ensuremath{\mathcal{M}}}
\newcommand{\sgeq}{\ensuremath{\succeq}}
\newcommand{\sleq}{\ensuremath{\preceq}}

\newcommand{\lkp}[1]{lk_{\mathcal{D}^+}(#1)}
\newcommand{\lk}[1]{lk_{\mathcal{D}}(#1)}

\tikzset{vertex/.style={circle, draw, fill=black!50},inner sep=0pt, minimum width=4pt}

\newcommand{\cst}[2]{[{#1}]_{#2}}
\newcommand{\mr}[2]{\overline{#1}_{#2}}
\newcommand{\End}[2]{\operatorname{end}_{#2}(#1)}

\title{A Deligne complex for Artin monoids}
\author{Rachael Boyd, Ruth Charney and Rose Morris-Wright}

\dedicatory{{Dedicated to the memory of Patrick Dehornoy.}}
\begin{document}

\subjclass[2010]{
	20F36 (primary), 
	20F55, 
	20M32, 
	20F65 (secondary). 
}
\keywords{Artin monoids, Artin groups, $K(\pi,1)$ conjecture.}

\thanks{Charney was partially supported by NSF grant DMS-1607616.}
\thanks{Boyd was partially supported by the London Mathematical Society Cecil King Scholarship.}

\begin{abstract}

In this paper we introduce and study some geometric objects associated to Artin monoids.  The Deligne complex for an Artin group is a cube complex that was introduced by the second author and Davis~\cite{CharneyDavis1995} to study the $K(\pi,1)$ conjecture for these groups.  Using a notion of Artin monoid cosets, we construct a version of the Deligne complex for Artin monoids.

We show that for any Artin monoid this cube complex is contractible. Furthermore, we study the embedding of the monoid Deligne complex into the Deligne complex for the corresponding Artin group. We show that for any Artin group this is a locally isometric embedding. In the case of FC-type Artin groups this result can be strengthened to a globally isometric embedding, and it follows that the monoid Deligne complex is CAT(0) and its image in the Deligne complex is convex.
We also consider the Cayley graph of an Artin group, and investigate properties of the subgraph spanned by elements of the Artin monoid. Our final results show that for a finite type Artin group, the monoid Cayley graph embeds isometrically, but not quasi-convexly, into the group Cayley graph. 

\end{abstract}

\maketitle

\section{Introduction} \label{Section: Introduction}

Artin groups, also known as Artin-Tits groups, are a broad class of groups whose presentations are encoded by labelled graphs. 
Given a simple graph $\G$ with a finite vertex set $S$ and edge set $E$, such that each edge $(s,t) \in E$ is labelled by an integer $m_{st}\geq 2$, we define the \textit{Artin group}, $\AG$, to be the group with presentation 

\[
\AG =\langle S \mid \underbrace{sts\ldots}_{\text{length }m_{st}}=\underbrace{tst\ldots}_{\text{length }m_{st}} \,\, \forall (s,t) \in E \rangle.
\] 
If there is no edge in $\G$ between $s$ and $t$ in~$S$, we say that $m_{st}=\infty$ and there is no relation between $s$ and $t$ in the presentation.

Artin groups are closely related to Coxeter groups. Given a graph $\G$ as above, the Coxeter group $\WG$ is the group whose presentation is the same as $\AG$ with the added relations $s^2=e$ for all $s\in S$, where $e$ is the identity element.  One important example of a Coxeter group is the symmetric group $S_n$. The symmetric group acts on $\R^n$ by permuting the coordinates.  Each transposition~$s_i$ exchanging the $i^{th}$ coordinate with the $i+1^{th}$ coordinate in $S_n$ is an involution, or reflection. The presentation of $S_n$ generated by~$S=\{s_1,\ldots, s_{n-1}\}$ is that of a Coxeter group. The corresponding Artin group is the braid group on $n$ strands. More generally, any Coxeter group can be realized as a discrete group generated by reflections on a finite dimensional vector space with respect to some inner product.  The Coxeter group is finite precisely when this inner product is positive definite.

The braid group is the prototypical example of a \textit{finite type} Artin group, an Artin group whose corresponding Coxeter group is a finite group.  The combinatorial structure of these groups was first studied by Garside \cite{Garside1969} who found a particularly nice solution to the word problem for these groups which has played a major role in the study of finite type Artin groups.  The notion of a \emph{Garside group} was later introduced by Dehornoy and Paris \cite{Dehornoy1999} to include other groups with a similar combinatorial structure.

However, the definition of Artin groups encompasses a class of groups much larger than only the finite type groups, and as a whole, this class is very poorly understood.  Particular types of Artin groups are well studied, including finite type, right-angled and FC-type Artin groups. However for general Artin groups, many basic questions remain unanswered.  For example, it is unknown whether they are torsion-free and whether they have solvable word problem, as well as many other properties.

Questions that seem intractable for Artin groups are often easier to solve in the monoid case. 
Given a labelled graph~$\G$ as above, the \emph{Artin monoid}~$\AG^+$ is defined to be the monoid given by the positive presentation of the Artin group: i.e.~elements in~$\AG^+$ are represented by words which use only positive powers of the generating set~$S$.
\[
A_\G^+ =\langle S \mid \underbrace{sts\ldots}_{\text{length }m_{st}}=\underbrace{tst\ldots}_{\text{length }m_{st}} \,\,  \forall (s,t) \in E \rangle^+.
\]
Some properties of the monoid are immediate, and solve questions that are still unknown for the group. For example, because all of the relations in the above presentation preserve the length of a word in terms of the standard generating set, there is a well defined length function for any given monoid element. Thus the word problem in the monoid is easily solved. 

The monoid also has the structure of a partially ordered set, with a relation determined by a prefix (or suffix) order on the elements. This poset is especially useful  in the finite type case, when it turns out the poset is a lattice, yet it can also be useful in the more general case. One goal of this work is to use the monoid and this poset to study the group in as much generality as possible.

One challenge in using the Artin monoid to study the Artin group is that very little is known about the relationship between the two. In \cite{Paris2002}, Paris shows that the natural map $\AG^+\to \AG$ is an injection, and even this result is highly non-trivial. 
In the current paper, we study geometric relationships between monoids and groups with the hope of building new tools to study both the monoid and the group.

\subsection{The monoid Deligne complex}

To date, the most effective approaches to studying infinite type (non-finite type) Artin groups have been geometric.  In particular, CAT(0) cube complexes have been a primary tool in the study of  right-angled Artin groups and, more generally, FC-type Artin groups.



The subgroup $A_T$ generated by a subset $T \subseteq S$ is called a special subgroup of $\AG$.  By a theorem of van der Lek \cite{VanderLek1983}, $A_T$ is isomorphic to the Artin groups associated to the (full) subgraph of $\G$ spanned by $T$.  An Artin group is called \textit{FC-type} if any subset $T$ that spans a clique in $\G$, generates a finite type Artin group.   FC-type Artin groups were originally defined by the second author and Davis in \cite{CharneyDavis1995}.  Building on work of Deligne, they define a cube complex whose vertices are given by cosets of finite type special subgroups of  $A_\G$.  Charney and Davis call this complex the modified Deligne complex, and it has since become known as the Deligne complex.

\begin{definition}\label{def:delignecomplex}
	Let $A_\G$ be an Artin group. The \textit{Deligne complex}~$\DG$ is the cube complex with vertex set all cosets $gA_T$, such that $A_T$ is finite type.  We partially order the vertices by inclusion of cosets. For any pair of vertices $gA_T \subset gA_{T'}$, the interval $[gA_T, gA_{T'}]$ spans a cube of dimension $|T' \smallsetminus T|$. 
\end{definition}

The action of $\AG$ on its cosets induces a cocompact action by isometries of $\AG$ on the Deligne complex~$\DG$. Note, however, that this action is not proper as the stabilizer of a vertex $gA_T$ is the subgroup $gA_Tg^{-1}$.

FC-type Artin groups are precisely those for which the standard cubical metric on this complex is CAT(0).  The ``FC" stands for Flag Complex, which comes from the flag condition required to show that a cube complex is CAT(0).  The Deligne complex has been used to show that FC-type Artin groups have many desirable properties. It was originally introduced to prove the $K(\pi,1)$  conjecture, which we discuss below.  In addition, it has been used to show that FC-type Artin groups have solvable word problem, are torsion-free and have finite virtual cohomological dimension, among other properties \cite{Altobelli1998, CharneyDavis1995, Godelle2007}.



In this paper, we use the Deligne complex to study the relation between the Artin monoid and the Artin group. The Deligne complex~$\DG$ is built from cosets of finite type special subgroups of $\AG$. In the first author's work \cite{Boyd2020}, an analogue to these cosets for the Artin monoid is defined and studied, in the setting of homological stability. In previous work of the second author \cite{Charney1999} similar cosets were used to define a `positive Deligne complex' to prove injectivity of the Artin monoid for 2-dimensional Artin groups. This has inspired the current work, in which we imitate the construction of the Deligne complex for the Artin monoid to produce a cube complex, $\DGP$, which we call the \emph{monoid Deligne complex}. We study geometric properties of~$\DGP$ and its relation with~$\DG$.

First, we show that for any Artin monoid $\AG^+$, the complex $\DGP$ is contractible. The analogous result for $\DG$ is known only for certain restricted classes of Artin groups.  (Indeed, this is one of the major open problems for a general Artin group.)

\begin{theorem*}{\ref{thm:contactible}}
	Let~$A_\G^+$ be an arbitrary Artin monoid. Then the cube complex~$\DGP$ is contractible.
\end{theorem*}

We then compare the geometry of  this new complex $\DGP$ to the full Deligne complex $\DG$.  Our main results are the following.

\begin{theorem*}{\ref{Thrm:Local Convexity}}  Let $\AG$ be any Artin group.  Then $\DGP$ embeds as a subcomplex of $\DG$ and the inclusion map $\iota: \DGP \to \DG$ is a locally isometric embedding.
\end{theorem*} 

This theorem applies to all Artin groups, but has important consequences when restricted to the FC-type case.

\begin{corollary*}{\ref{Cor:DGPisCAT(0)}}
	If $\AG$ is an FC-type Artin group, then the inclusion map $\iota: \DGP \to \DG$ is an isometric embedding, hence $\DGP$ is CAT(0) and its image is convex in $\DG$.
\end{corollary*}

\subsection{The monoid Cayley graph}

The Deligne complex considers group elements `up to finite type cosets', which in many cases reduces problems to the well-studied finite type case. We can however, consider the whole set of group elements by studying the Cayley graph of the group.
We consider the Cayley graph of a given Artin group and study the properties of the subgraph spanned by elements in the monoid, which we call the monoid Cayley graph.  

Identifying elements of the group and monoid with vertices in the Cayley graph and monoid Cayley graph, we get a metric on $\AG$ and $\AG^+$ given by minimal path lengths in the corresponding graph. Using the Cayley graph of $\AG$ defined with respect to a particular finite generating set called $\calM$, the set of minimal elements, we show the following.

\begin{prop*}{\ref{Prop:IsoEmdbedding}}
	Suppose~$\AG$ is a finite type Artin group, then with respect to the generating set $\calM$, the associated monoid $\AG^+$ embeds isometrically in the group $\AG$.
\end{prop*}

On the other hand, we give an example showing that this inclusion is not in general convex or even quasi-convex.  That is, geodesics in the Cayley graph of $\AG$ connecting two monoid vertices, need not stay uniformly bounded distance from the monoid subgraph.

\subsection{Motivation: the~$K(\pi,1)$ conjecture}

Much of this paper focuses on establishing a geometric relationship between an Artin group and its corresponding monoid, but Artin monoids and their geometry are also interesting to study in their own right.

Much of the early work on Artin groups focussed on solving a conjecture formulated in its current form by Arnol'd, Brieskorn, Pham and Thom. First we consider the case of finite type Artin groups. In this case, given a defining graph~$\G$, one can associate to the (complexified) action of the Coxeter group $W_\G$ on $\C^n$ a hyperplane complement obtained by removing all of the hyperplanes fixed by some reflection.  We will denote this hyperplane complement by~$\mathcal{H}_\G$. The  group~$W_\G$ acts freely on~$\mathcal{H}_\G$ and the corresponding quotient $\mathcal{H}_\G/W_\G$ has as its fundamental group the Artin group~$A_\G$.  In fact, this was the original motivation for the definition of Artin groups by Brieskorn~\cite{Brieskorn1971}.  
For example, in the case that $W_\G$ is the symmetric group on $n$ letters, $\mathcal{H}_\G/ W_\G$ is the configuration space of $n$ (unordered) distinct points in the complex plane, and the braid group can be naturally identified with the fundamental group of this space.  

In work of Deligne~\cite{Deligne1972}, the universal cover of $\mathcal{H}_\G$ (and hence also of $\mathcal{H}_\G/W_\G$) is shown to be contractible and it follows that~$\mathcal{H}_\G/W_\G$ is an Eilenberg-Maclane space for~$A_\G$, otherwise known as a~$K(A_\G,1)$ space or classifying space~$BA_\G$. His proof, however, applies only to the finite type case.
For infinite type Artin groups, there is an analogue of the hyperplane complement formulated by Vinberg \cite{Vinberg1971} by restricting to an open cone in $\C^n$.  (For a more detailed description see Davis \cite{Davis2008}, notes by Paris \cite{Paris2014} and the introduction of \cite{Charney2007}.)  We again denote this hyperplane complement  by~$\mathcal{H}_\G$.   Van der Lek \cite{VanderLek1983} showed  that the fundamental group of $\mathcal{H}_\G/W_\G$ is isomorphic to the Artin group $\AG$ for any $G$.  The \emph{$K(\pi,1)$ conjecture} states that an analogue of Deligne's theorem holds for all Artin groups, that is, the universal cover of $\mathcal{H}_\G$ is contractible.    This conjecture is open in general, but is known to hold for many classes of Artin groups, including finite type~\cite{Deligne1972}, FC-type~\cite{CharneyDavis1995}, affine type~\cite{PaoliniSalvetti2018}, and 2-dimensional Artin groups \cite{Hendriks1985}

The~$K(\pi,1)$ conjecture has been rephrased in many ways. Charney and Davis~\cite{CharneyDavis1995} showed that the Deligne complex of Definition~\ref{def:delignecomplex} is homotopy equivalent to the universal cover of $\mathcal{H}_\G$ for any Artin group and thus contractibility of the Deligne complex would prove the~$K(\pi,1)$ conjecture.  Most of the cases for which the conjecture is known were proved in this way.  

There are also several other approaches.  Notably, Salvetti constructed a finite dimensional CW complex~$\operatorname{Sal}_\G$ homotopy equivalent to~$\mathcal{H}_\G$~\cite{Salvetti1994} (see also \cite{CharneyDavis1995a}),  so  proving that  this complex is aspherical would also prove the~$K(\pi,1)$ conjecture. 

In 2006 Dobrinskaya proved that the quotient $\mathcal{H}_\G/W_\G$ has the same homotopy type as~$BA_\G^+$, the classifying space of the Artin monoid~\cite{Dobrinskaya2006}. This was later reproven by Ozornova \cite{Ozornova2017} and Paolini \cite{Paolini2017}. Recall that, contrary to classifying spaces of groups, the classifying space of a monoid can exhibit any connected homotopy type~\cite{McDuff1979}. It follows that the~$K(\pi,1)$ conjecture is true for $\AG$, if and only if the natural map~$BA_\G^+\to BA_\G$ is a homotopy equivalence. We give a diagrammatic depiction of some of the known~$K(\pi,1)$ conjecture equivalences in the digram below. Finding a proof which confirms any question mark shown in the diagram would in turn prove the~$K(\pi,1)$ conjecture. Some of our results (Theorem~\ref{thm:contactible} and part of Theorem~\ref{Thrm:Local Convexity}) on the monoid Deligne complex are shown in red.
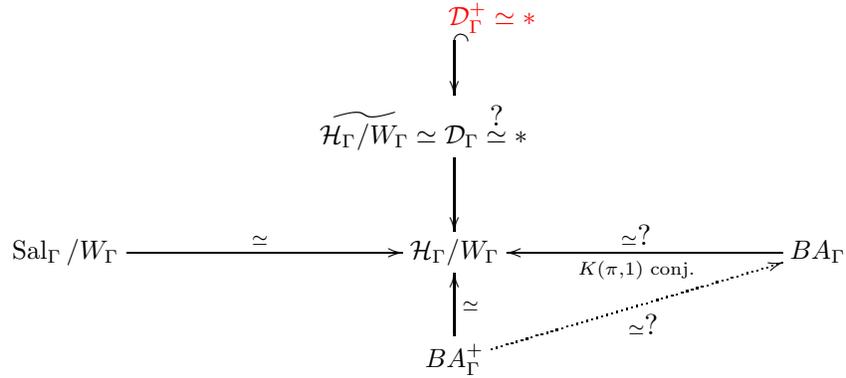
\[
\xymatrix@C=7em{
	& {\color{red} \hspace{1cm} \DGP\simeq \ast \ar@[red]@{^{(}->}[d]}&\\
	&\widetilde{\mathcal{H}_\G/W_\G}\simeq \DG \overset{\text{\large{?}}}{\simeq}\ast\qquad \,\, \ar[d]\\
	\operatorname{Sal}_\G/W_\G\ar[r]^{\simeq}&\mathcal{H}_\G/W_\G & BA_\G\ar[l]^{K(\pi,1) \text { conj.}}_{\simeq \text{\large{?}}}\\
	&BA^+_\G \ar[u]_{\simeq} \ar@{.>}[ru]_{\simeq \text{\large{?}}}\\
}
\]

The initial definition of Artin groups arose from their relation to hyperplane complements: the group encodes some of the homotopy information about the hyperplane complement, namely the fundamental group. The~$K(\pi,1)$ conjecture then implies that the Artin group in fact encodes \emph{all} of the homotopy information, and this has been shown to be true in many cases. However, from Dobrinskaya's work, we know that the Artin monoid  encodes all of the homotopy information for any $\G$.  This has motivated our study of these monoids and the related construction of a monoid Deligne complex.

\subsection{Discussion and further questions}

We believe that this work opens the door to many further questions on the geometry of Artin monoids and the relationship with their corresponding Artin group. Readers who have experience with Artin groups or CAT(0) geometry will no doubt be able to think of further questions in this direction, we will discuss a few below that seem natural to us.

There are a number of other geometric structures that have been used to study different classes of Artin groups.  For example, in \cite{CharneyMorrisWright2019}, the second and third authors use another cube complex, called the clique cube complex, to show that most irreducible infinite type Artin groups have trivial center and are acylindrically hyperbolic.   The clique cube complex is defined using cosets of special subgroups corresponding to cliques in the defining graph. Unlike the Deligne complex, the clique cube complex is CAT(0) for all Artin groups.  This complex has also been used to show that many questions about general Artin groups can be reduced to the case of Artin groups whose defining graph is a single clique \cite{GodelleParis2012a, CharneyMorrisWright2019}. One can define a \emph{monoid clique cube complex}, in the same vein as the monoid Deligne complex, using the monoid cosets of special subgroups corresponding to cliques. 
This inspires the next question, which we leave broad.

\begin{question}
	Which of our results for the monoid Deligne complex hold true for the monoid clique cube complex?
\end{question}

 Our results are particularly strong in the FC-type case because in that case, the cubical metric on the Deligne complex is CAT(0).  There is another metric on the Deligne complex, known as the Moussong metric, that is conjectured to be CAT(0) for all Artin groups.  This is known to be the case when $\AG$ is 2-dimensional~\cite{CharneyDavis1995}, and more generally when it is locally reducible (see \cite{Charney2000}). The second author shows in \cite{Charney1999}  that for 2-dimensional Artin groups,~$\DGP\hookrightarrow\DG$ is an isometric embedding.  This result can been seen to extend to locally reducible Artin groups using results on spherical joins from the Appendix of \cite{CharneyDavis93}. The following question is a natural generalization of these results.

\begin{question}  In general, is the embedding $\DG^+ \hookrightarrow \DG$ locally isometric with respect to the Moussong metric?
\end{question}

Considering our results on the monoid Cayley graph, there are many potential strengthenings one might desire.  The following question seems reasonable to us, 
in light of the fact that in the FC-type case, we have shown that~$\DGP$ embeds isometrically into~$\DG$.

\begin{question}
	In the FC-type case, does the monoid Cayley graph isometrically embed into the Cayley graph of the full group with respect to either the standard generating set $S$ or the generating set of minimal elements $\calM$?
\end{question}

We finish with a very ambitious question related to our main motivating problem, the~$K(\pi,1)$ conjecture. One can take translates of the (contractible) subcomplex~$\DGP$ in~$\DG$, given by the action of the Artin group on~$\DG$. Taking `enough' of these translates will cover~$\DG$.

\begin{question}
	Can one use the covering of~$\DG$ by translates of~$\DGP$ to prove contractibility of~$\DG$, and hence prove the~$K(\pi,1)$ conjecture, for some new classes of Artin groups?
\end{question}

One might begin by considering this question in the FC-type case -- where the contractibility of~$\DG$ is known -- as this may provide further insight into the question for more general~$\DG$.

\subsection{Outline}
In \Cref {Section:Background} we review background material on Artin groups and their monoids as well as the geometry of CAT(0) cube complexes. In \Cref{Section:Definition}, we define the monoid analogue of cosets, as originally described in \cite{Boyd2020}, and use these cosets to define a monoid version of the Deligne complex, denoted $\DGP$. In \Cref{Section:Contractibility}, we prove that the monoid Deligne complex~$\DGP$ is always contractible.  In \Cref{Section:Convexity}, we investigate the geometric properties of the embedding of $\DGP$ into the Deligne complex $\DG$.   In \Cref{Section:CayleyGraph}, we consider the Cayley graph of an Artin group and the subgraph spanned by monoid elements. Specifically we consider questions about the convexity of this subgraph. 

\subsection{Acknowledgements} 
The first author would like to thank the Max Plank Institute for Mathematics in Bonn for its support and hospitality.

This project began during a two-week summer school at the Institut des Hautes \'Etudes Scientifiques.  All three authors would like to thank IHES for their hospitality.

\section{Background} \label{Section:Background}

				In this section we collect basic facts and lemmas that we require in the rest of the paper.
				
				\subsection{Artin groups and monoids} \label{Subsection:ArtinMonoids}
				In this section we recall some basic definitions and properties of the groups and monoids we work with. General references for readers are Paris~\cite{Paris2014}, \cite{Michel1999} and Brieskorn and Saito \cite{BrieskornSaito1972} (an English translation of this paper also exists \cite{BrieskornSaitoTranslation}).
				
				Consider a finite set~$S$ and a finite simple graph~$\G$ with vertex set~$S$ and edge set, $E$, where each edge $(s,t)\in E$ is labelled by an integer greater $m_{st}\geq 2$. If there is no edge between $s$ and $t$ we set $m_{st}=\infty$.
				
				\begin{defn}
					Given a finite generating set~$S$ and corresponding graph~$\G$ as above, we define the \emph{Artin group}~$A_\G$ to be the group with presentation
					\[
					A_\G =\langle S \mid \underbrace{sts\ldots}_{\text{length }m_{st}}=\underbrace{tst\ldots}_{\text{length }m_{st}} \,\, \forall (s,t) \in E \rangle.
					\]
				\end{defn}
				
				For each graph~$\G$ there also exists a corresponding Coxeter group~$W_\G$, given by adding the relations~$s^2=e$ for all $s\in S$ to the presentation for~$A_\G$. The finite Coxeter groups were classified by Coxeter~\cite{Coxeter1933}, and if~$W_\G$ is finite, we say that~$A_\G$ is a \emph{finite type} Artin group (note since all generators have infinite order an Artin group is never itself finite).
				
				\begin{defn}
					Given a finite generating set~$S$ and corresponding graph~$\G$, the Artin monoid~$\AG^+$ is defined to be the monoid given by the positive presentation of the Artin group: i.e.~elements in~$\AG^+$ are represented by words which use only positive powers of the generating set~$S$.
					\[
					A_\G^+ =\langle S \mid \underbrace{sts\ldots}_{\text{length }m_{st}}=\underbrace{tst\ldots}_{\text{length }m_{st}} \,\, \forall (s,t) \in E\rangle^+.
					\]
				\end{defn} 
				
				Note here that~$A_\G$ is the group completion of~$A_\G^+$. Artin monoids appear in much of the work on Artin groups. For instance in the seminal work of Deligne~\cite{Deligne1972} and Brieskorn and Saito~\cite{BrieskornSaito1972} properties of Artin monoids, and the relationship between the monoids and the groups, play a huge role. In particular, a key property of finite type Artin groups is that every element~$w$ can be written as~$w=ab^{-1}$ for~$a$ and~$b$ in the corresponding Artin monoid~\cite{Garside1969, Deligne1972}.
				
				\begin{defn}
					Given a subset~$T\subseteq S$, the full subgraph of~$\G$ spanned by the vertex  subset~$T$ defines an Artin group in its own right, which is a subgroup of~$A_\G$. We denote this subgroup~$A_T$, and call such subgroups~\emph{special subgroups}. If a special subgroup~$A_T$ is a finite type Artin group, we call it a \emph{spherical}  or \textit{finite type} subgroup. In this setting it is useful to keep track of which subsets~$T\subseteq S$ give rise to spherical subgroups. We define
					\[
					\mathcal{S}^f= \{ T \subseteq S \mid \textrm{$A_T$  is finite type} \}.
					\]
					We similarly define~$A_T^+$ to be the submonoid of~$A^+_\G$ corresponding to the subgraph of~$\Gamma$ spanned by~$T$.
				\end{defn}

				Many of the technical lemmas in this work involve manipulation of words or elements in a given Artin group or monoid. The following definitions and lemmas provide us with a tool-kit with which to compare or manipulate these elements.
				
				\begin{defn}\label{def:length}
					Let~$\alpha$ be an element in~$A_\G^+$. We define the length of~$\alpha$ with respect to the standard generating set $S$ by
					\[
					l(\alpha)=k \text{ if $\alpha$ can be written as the word }\alpha=s_1\ldots s_k \text{ for } s_i \in S.
					\]
				\end{defn}
				The length function~$l: A_\G^+\to \N$ is a well-defined monoid homomorphism. It is independent of the word chosen to represent the element~$\alpha$, since the relations in the Artin monoid equate words of the same length, and it is a homomorphism since multiplication in the monoid corresponds to addition of lengths.

				\begin{lemma}[{\cite[Proposition 2.3]{BrieskornSaito1972}}\cite{Michel1999}]
					Artin monoids satisfy {right} cancellation: that is if~$a,b,c \in A^+_\G$ satisfy $ac=bc$ then it follows that~$a=b$. {Likewise for left cancellation.}
				\end{lemma}

				There are two partial orderings on the Artin monoid that will play a key role in the forthcoming arguments.
				
				\begin{defn}\label{def:partialorderings}
					For~$a,b \in A_\G^+$, let~$\sleq_L$ denote the partial ordering on~$\AG^+$ defined by~$a \sleq_L b$ if~$b=ac$ for some~$c \in \AG^+$. In this case we say~$a$ is a \emph{left divisor} of~$b$, and~$b$ is a \emph{left multiple} of~$a$.  Similarly, let~$\sgeq_R$ denote the partial ordering on~$\AG^+$ defined by~$b \sgeq_R a$ if~$b=ca$ for some~$c \in \AG^+$. In this case we say~$a$ is a \emph{right divisor} of~$b$, and~$b$ is a \emph{right multiple} of~$a$.
				\end{defn}
				
				{Note that in the definitions above, `left' and `right'  refers to the choice of ordering, $\sleq_L$ or $\sgeq_R$.  We caution that some authors use these terms differently.}
				
				\begin{defn}\label{def:cms-and-gcms}
					Given a subset~$X$ of an Artin monoid~$A_\G^+$, we say that $b\in A_\G^+$ is a \emph{common left multiple} for~$X$  if~$x\sleq_L b$ for all~$x\in X$.     We say that $a\in A_\G^+$ is a \emph{common left divisor} for~$X$  if~$a\sleq_L x$ for all~$x\in X$.  Similarly for common right multiples and divisors.
				\end{defn}
				
				\begin{lemma}[{\cite[Proposition~4.1]{BrieskornSaito1972}}]
					If a finite subset~$X$ of an Artin monoid has a common left/right multiple, then it has a (unique) least common left/right multiple, which we denote by $\lcm_L(X)$, or $\lcm_R(X)$.
				\end{lemma}

				\begin{lemma}[{\cite[Proposition~4.2]{BrieskornSaito1972}}]   Any finite subset $X$ of an Artin monoid has a (unique) greatest common left/right divisor, which we denote by $\gcd_L(X)$, or $\gcd_R(X)$.
				\end{lemma}
				
				
				A key property of finite type Artin groups is that the generating set $S$ has a common left multiple and a common right multiple and these common multiples are equal.  Moreover, conjugation by this common multiple permutes the elements of $S$.

				\begin{defn}\label{def:garsideelement}
					Let~$A_\G$ be a finite type Artin group. The \emph{Garside element}~$\Delta_S\in \AG^+$ is the unique element satisfying $\Delta_S=\lcm_L(S)=\lcm_R(S)$.  
				\end{defn}

				Gaussian and Garside monoids and groups were introduced by Dehornoy and Paris in 1999, as generalisations of finite type Artin monoids and groups \cite{Dehornoy1999}. Finite type Artin groups are examples of Garside groups, and thus they have a~\emph{Garside structure}, based on the existence of the Garside element, which we exploit  throughout this work.  In particular, the Garside element allows for a direct correspondence between group elements and monoid elements.

				\begin{lemma}\label{lem:garsideform}
					Given a finite type Artin group~$A_\G$ and an element~$\alpha \in A_\G$, there exists~$m\in A^+_\G$ and~$k\in\Z$ such that $\alpha=m\Delta_S^k$.  Moreover, this decomposition is unique if we require that~ $m\nsucceq_R\Delta_S$.
				\end{lemma}
				
				We emphasize that Garside elements  exist only for finite type Artin groups.  Indeed, Brieskorn and Saito \cite{BrieskornSaito1972} prove that for $A_\G$ of infinite type,  no element in the monoid can be a common right or left multiple of $S$.
				
				\begin{defn}
					Let~$\alpha \in A_\G^+$. Define the subset $T_\alpha\subseteq S$ to be the set of all generators~$s\in S$ such that $\alpha\sgeq_R s$, that is
					\[
					s\in T_\alpha \iff \alpha =\beta s \text{ for some }\beta \in A_\G^+.
					\]
				\end{defn}

				\begin{lemma}[\cite{BrieskornSaito1972}]\label{lem:endgenfinitetype}
					For any~$\alpha \in A_\G^+$, the subset~$T_\alpha\subseteq S$ is finite type, i.e.~$T_\alpha \in \SF$.
				\end{lemma}
				
				\bigskip
				
				\subsection{CAT(0) cube complexes} \label{Subsection:CAT(0)} 
				
				In addition to the combinatorial structure of the Artin monoid and Artin group, we will be interested in the geometric structure of their associated Deligne complexes. In this section we review some geometric notions that will be used in later sections.  For more details and proofs, see \cite{BridsonHaefliger2011}.
				
				Let $(X,d)$ be a geodesic metric space, that is, a metric space in which any two points $x,y$ are connected by a path of length $d(x,y)$.  Such a path is called a \emph{geodesic} from $x$ to $y$. 
				Let $T(a,b,c)$ denote a triangle in $X$ with vertices $a,b,c$ and geodesic edges $[a,b], [b,c], [c,a]$.
				A \emph{comparison triangle} is a triangle $T(a',b',c')$ in the Euclidean plane with the same edge lengths.  The CAT(0) condition states that triangles in $X$ are ``at least as thin" as their comparison triangles in $\mathbb E^2$.  More precisely,
				
				\begin{definition} A geodesic metric space $X$ is \emph{CAT(0)} if for any geodesic triangle $T(a,b,c)$  in $X$ and any points $x \in [a,b]$, $y \in [b,c]$, the corresponding points in a comparison triangle $x' \in [a',b']$, $y' \in [b',c']$ satisfy 
					$$d_X(x,y) \leq d_{\mathbb E^2}(x',y').$$
					We say $X$ is \emph{locally CAT(0)} if every point in $X$ has a neighbourhood which is CAT(0).
				\end{definition}
				
				CAT(0) spaces satisfy many nice properties. Here are a few well-known facts. 
				\begin{itemize}
					\item If $X$ is CAT(0), then any two points in $X$ are connected by a unique geodesic.
					\item  If $X$ is CAT(0), then it is contractible.
					\item  If $X$ is locally CAT(0) and simply connected then it is (globally) CAT(0).
					\item  Any locally geodesic path in a CAT(0) space is  a geodesic.
				\end{itemize}
				
				A particularly useful class of CAT(0) spaces are CAT(0) cube complexes (CCCs), both because they are easy to construct and because they come with a nice combinatorial structure in addition to their geometric structure.  We recall a few basics of CCCs here and refer the reader to \cite{Haglund2008, Sageev2012} for more details.
				
				A cube complex is a space obtained by gluing together a collection of standard Euclidean cubes, $[0,1]^n$, of varying dimensions, via isometries of faces.  Let $X$ be such a complex with the induced path metric $d$.  The link of a vertex $v$ in $X$, denoted $lk_X(v)$, (or simply $lk(v)$  when $X$ is understood)  is the simplicial complex with a $k$-simplex for each $k+1$ cube $C$ containing $v$. Viewing this simplex as the unit tangent space of $v$ in $C$, we can identify this simplex with a quadrant in the unit sphere $\mathbb S^{k}$.  This gives rise to  a natural piece-wise spherical metric on $lk(v)$ with all edges of length $\frac{\pi}{2}$.  
				
				Identifying $lk(v)$ with the unit tangent space of $v$ in $X$, distances in $lk(v)$ correspond to angles between tangent vectors.  In particular, a path $p$ passing through $v$ is locally geodesic at $v$ if and only if its incoming and outgoing tangent vectors have distance at least $\pi$ in $lk(v)$.  A similar piece-wise spherical metric can be put on the link of a point $x$ in a higher dimensional face and once again, a path in $X$ passing through $x$ is locally geodesic at $x$ if and only if its incoming and outgoing tangent vectors have distance at least $\pi$ in this $lk(x)$.
				
				Using this fact, Gromov showed that a certain combinatorial condition on these links was sufficient to determine whether $X$ is locally CAT(0).  
				
				\begin{definition}  A simplicial complex $L$ is a \emph{flag complex} if every set of $k$ vertices in $L$ that are pairwise joined by edges, span a $k+1$ simplex.  In particular, a flag complex is completely determined by its one-skeleton. 
				\end{definition}
				
				Gromov's condition states,
				
				\begin{theorem}   A cube complex $X$ is locally CAT(0) if and only if for every vertex $v$ in $X$, $lk_X(v)$ is a flag complex.  
				\end{theorem}
				
				In addition to their geometry, CAT(0)  cube complexes come with a combinatorial structure given by \emph{hyperplanes}.  These are codimension-one subspaces, made up of midplanes of cubes, that divide the complex into two components.  Much of the theory of CCCs depends on understanding the interplay between hyperplanes.  Moreover, in addition to the CAT(0) metric on a CCC $X$, there is another metric  known as the $\ell^{(1)}$-metric.  This is usually only applied to the 1-skeleton of $X$ and it is defined to be the minimal length of an edge path between two points in $X^{(1)}$.  This metric is not CAT(0) and there may be many minimal length edge paths between two points.  Nevertheless, the CAT(0) geodesic and the $\ell^{(1)}$-geodesics between two point $x,y$ are related by the fact that they cross exactly the same hyperplanes, namely the hyperplanes that separate $x$ from $y$.   
				
				\begin{definition}  Let $X$ be a CAT(0) cube complex.  For two vertices $v,w$ in $X$, the subcomplex spanned by the $\ell^{(1)}$-geodesics from $v$ to $w$ is called the \emph{cubical convex hull} of $x,y$.  With respect to the CAT(0) metric, it is the smallest convex subcomplex in $X$ containing the geodesic from $v$ to $w$.
				\end{definition}
				
				For example, dividing the Euclidean plane $\R^2$ into unit squares with vertices in $\Z^2$, for any two vertices $x,y$ the CAT(0) geodesic is the straight line connecting them and their cubical convex hull is the rectangle with this line as its diagonal.

\section{Definition of $\DGP$} \label{Section:Definition}

	Recall from \Cref{def:delignecomplex} that the Deligne complex $\DG$ is the cube complex associated to the partially ordered set of cosets $aA_T, T \in \SF$.  To define the monoid Deligne complex $\DGP$, we will need to define and understand the properties of monoid cosets. 	
	
	\begin{defn}\label{def:monoid coset}
		Given~$A_T^+$ a submonoid of~$A^+$, consider the relation $\sim_T$ on $\AG^+$ given by
		$$
		\alpha \sim_T \beta \iff \alpha t=\beta t' \text{ for some }t\text{ and }t'\text{ in }A_T^+.
		$$
		\noindent The relation~$\sim_T$ is symmetric and reflexive. Let $\approx_T$ be the transitive closure of $\sim_T$. That is, $\alpha \approx_T \beta$ if there is a chain of elements~$\tau_i$ in $\AG^+$ such that:
		$$
		\alpha \sim_T \tau_1 \sim_T \tau_2 \sim_T \cdots \sim_T \tau_k \sim_T \beta
		$$
		\noindent for some $k$. Denote the equivalence class of $\alpha$ under the relation $\approx_T$ as $\cst{\alpha}{T}$.
	\end{defn}

	It is shown in \cite{Boyd2020} that for any $\alpha \in \AG^+$, the set of right divisors of $\alpha$ that lie in the submonoid $A_T^+$ has a unique maximal element which we denote by $\End{\alpha}{T}$, and we can factor $\alpha$ as 
	$$\alpha = \mr{\alpha}{T}\cdot \End{\alpha}{T}.$$
	
	\begin{lemma}\label{lem:representative}  Let $\alpha \in A^+$.  Then
		for any $\beta \approx_T \alpha$, $\mr{\beta}{T} = \mr{\alpha}{T}$ so
		$$
		\cst{\alpha}{T}=\{\beta \in A^+ \mid \beta  = \mr{\alpha}{T} t \text{ for some }t \in A^+_T\}.
		$$
		Moreover, $\mr{\alpha}{T}$  is the unique minimal length representative of $\cst{\alpha}{T}$.  
	\end{lemma}	
	
	\begin{proof}
		The first statement is Lemma 5.21 in \cite{Boyd2020}.  By definition, $\mr{\alpha}{T}$ is the smallest left divisor of $\alpha$ lying in $\cst{\alpha}{T}$.  Since these minimal left divisors are the same for every $\beta \in \cst{\alpha}{T}$, the second statement follows.
	\end{proof}

	We now establish some useful properties of these monoid cosets.

	\begin{lemma}\label{lem:cosets} Let $\alpha \in A^+$ and $T_1,T_2 \subseteq S$.  Then
		\begin{enumerate}
			\item $\cst{\alpha}{T_1} \subset \cst{\alpha}{T_2} $ if and only if $T_1 \subset T_2$.
			\item $\cst{\alpha}{T_1} \cap \cst{\alpha}{T_2}=\cst{\alpha}{T_1 \cap T_2}$.
		\end{enumerate}
	\end{lemma}
	
	\begin{proof}
		(1) $T_1 \subset T_2$ implies $\cst{\alpha}{T_1} \subset \cst{\alpha}{T_2}$ is clear from the definition of the equivalence relation.  Conversely, suppose $\cst{\alpha}{T_1} \subset \cst{\alpha}{T_2}$.  Then $\mr{\alpha}{T_1}=\mr{\alpha}{T_2} t$ for some $t \in A^+_{T_2}$.  So for any $s \in T_1$, the element $\beta = \mr{\alpha}{T_1} s = \mr{\alpha}{T_2} t s$ lies in $\cst{\alpha}{T_1} \subset \cst{\alpha}{T_2}$.  Thus we can also write $\beta = \mr{\alpha}{T_2} t'$ for some $t' \in A^+_{T_2}$. Cancelling $\mr{\alpha}{T_2}$ we see that $t s = t' \in T_2$, and conclude that $s \in T_2$.
		
		(2)  The inclusion~ $\cst{\alpha}{T_1} \cap \cst{\alpha}{T_2} \supseteq \cst{\alpha}{T_1 \cap T_2}$ follows from part (1), so it remains to show that $\cst{\alpha}{T_1} \cap \cst{\alpha}{T_2}\subseteq \cst{\alpha}{T_1 \cap T_2}$. That is, given~$\beta$ such that~$\mr{\beta}{T_1}=\mr{\alpha}{T_1}$ and~$\mr{\beta}{T_2}=\mr{\alpha}{T_2}$ we wish to show~$\mr{\beta}{T_1 \cap T_2}=\mr{\alpha}{T_1 \cap T_2}$.

		{\bf Claim}: For any word~$\gamma$ in~$A^+$,~$\mr{\gamma}{T_1 \cap T_2}$ is the least common left-multiple of~$\mr{\gamma}{T_1}$ and~$\mr{\gamma}{T_2}$.
		
		Given the claim, it follows that if $\mr{\beta}{T_1}=\mr{\alpha}{T_1}$ and~$\mr{\beta}{T_2}=\mr{\alpha}{T_2}$ then their least common left-multiples are also equal, that is~$\mr{\beta}{T_1 \cap T_2}=\mr{\alpha}{T_1 \cap T_2}$. 
		
		To prove the claim, recall that for any $T$, $\gamma = \mr{\gamma}{T} \End{\gamma}{T}$ where $\End{\gamma}{T}$ is the greatest right divisor of $\gamma$ contained in $A^+_T$.  Thus, $\mr{\gamma}{T_1 \cap T_2}$ is the least common left-multiple of~$\mr{\gamma}{T_1}$ and~$\mr{\gamma}{T_2}$ if and only if $\End{\gamma}{T_1 \cap T_2}$ is the greatest common right divisor of~$\End{\gamma}{T_1}$ and~$\End{\gamma}{T_2}$.  Denote this greatest common right divisor by $g$.  
		
		Since  $g$ is a right divisor of $\End{\gamma}{T_1} \in A_{T_i}^+$  for $i=1,2$, it is a right divisor of $\gamma$ that lies in $A_{T_1 \cap T_2}^+$.  So by definition, $\End{\gamma}{T_1 \cap T_2} \sgeq_R g$. Conversely, any right divisor of $\gamma$ in $A_{T_1 \cap T_2}^+$ is also a right divisor in $A_{T_i}^+$,  so  
		$$\End{\gamma}{T_1 \cap T_2}=\End{\End{\gamma}{T_i }}{T_1 \cap T_2}.$$
		It follows that $\End{\gamma}{T_i} \sgeq_R \End{\gamma}{T_1 \cap T_2 }$ for both $i=1,2$, so $g \sgeq_R \End{\gamma}{T_1 \cap T_2}$ since~$g$ is the greatest common right divisor.  This proves the claim.
	\end{proof}

	In the discussion that follows, we will mostly be interested in cosets associated to subsets $T \in \SF$. In this case, the monoid cosets are closely related to the groups cosets. We view the Artin monoid~$A_\G^+$ as a sub-monoid of the Artin group~$A_\G$, and throughout use the highly non-trivial result, due to Paris \cite{Paris2002}, that the Artin monoid injects into the Artin group.
	
	\begin{lemma}\label{lem:group cosets}
		Let~$\alpha \in A^+_\G$. If $T \in \SF$, then $\cst{\alpha}{T} = \alpha A_T \cap A_\G^+$. 
	\end{lemma}
	
	\begin{proof} It is clear from the definition that $\cst{\alpha}{T} \subseteq \alpha A_T \cap A_\G^+$.  For the reverse inclusion, let $\beta \in \alpha A_T \cap A_\G^+$, so $\beta = \alpha g$ for some $g \in A_T$.  Since $T \in \SF$, $g$ can be written in the form $g=cd^{-1}$ where $c,d \in A_T^+$.  Thus $\beta d = \alpha c$ in the monoid, and we conclude that $\beta \in \cst{\alpha}{T}$. 
	\end{proof}
	
	\begin{remark}  In the the case that $T_1$ and $T_2$ are in $\SF$, there is an alternate proof of Lemma \ref{lem:cosets}(2).  Namely, 
		it follows from \cite{VanderLek1983} that for any $T_1,T_2$, the corresponding special  subgroups satisfy $A_{T_1} \cap A_{T_2} = A_{T_1 \cap T_2}$ and hence also
		$\alpha A_{T_1} \cap \alpha A_{T_2} = \alpha A_{T_1 \cap T_2}$.  Intersecting both sides with $A^+$ gives the desired equality.  
	\end{remark}
	
There is a natural partial ordering on the set of monoid cosets $\cst{\alpha}{T}$ given by inclusion.  Any two cosets which are related in this ordering have a common representative,  $\cst{\alpha}{T_1} \subset \cst{\alpha}{T_2}$.  By Lemma \ref{lem:cosets}, for such a pair, the interval $[\cst{\alpha}{T_1}, \cst{\alpha}{T_2}]$ consists of the set of cosets $\cst{\alpha}{T}$ with $T_1 \subseteq T \subseteq T_2$.

	\begin{definition}\label{Def:Deligne monoid complex}
		We define the \emph{monoid Deligne complex}~$\DGP$ to be the cube complex with vertices given by~$\cst{\alpha}{T}$ for~$\alpha \in A^+$ and~$T \in \mathcal{S}^f$,  and for~$T_1\subset T_2$,  the interval~$[\cst{\alpha}{T_1}, \cst{\alpha}{T_2}]$ spans a cube of dimension~$|T_2 \smallsetminus T_1|$. 
	\end{definition}
	
	The Artin monoid $\AG^+$ acts on $\DGP$ by left multiplication of cosets, that is, $\beta \cdot \cst{\alpha}{T} = \cst{\beta\alpha}{T}$.  This clearly preserves inclusions, and hence maps cubes to cubes.

	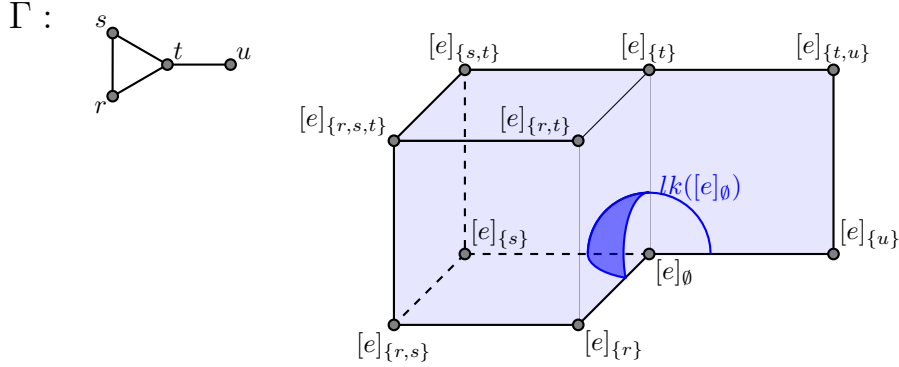
\begin{figure}
		
		\begin{tikzpicture}[thick, scale=.7]
		
		
		\begin{scope}[scale=1.2,shift={(-5,3)}]
		\draw (160:2.3) node[]{{\LARGE $\G:$}};
		
		\draw (0,0)--(1,0);
		\draw[rotate =150] (0,0)--(1,0)--(1/2,0.87)--(0,0);
		
		\draw (0,0) node[vertex,label={above right: $t$}]{};
		\draw[rotate =150] (1,0) node[vertex,label={above left:$s$}]{};
		\draw (1,0) node[vertex,label={above right:$u$}]{};
		\draw[rotate =210] (1,0) node[vertex,label={below left: $r$}]{};
		
		\end{scope}
		\begin{scope} [scale=3.5, shift={(.9,0)}]
		
		\coordinate (E) at (0,0,0);
		\coordinate (D) at (1,0,0);
		\coordinate (A) at (0,0,1);
		\coordinate (C) at (0,1,0);
		\coordinate (B) at (-1,0,0);
		\coordinate (CD) at (1,1,0);
		\coordinate (BC) at (-1,1,0);
		\coordinate (AC) at (0,1,1);
		\coordinate (AB) at (-1,0,1);
		\coordinate (ABC) at (-1,1,1);

		\filldraw[fill=blue!10] (E)-- (D)--(CD)--(C)-- cycle;
		\filldraw[fill=blue!10]  (E)--(B)--(BC)--(C); 
		\filldraw[fill=blue!10] (E)--(A)--(AC)-- (C); 
		\filldraw[fill=blue!10] (AC)--(ABC)--(AB)-- (A); 
		\filldraw[fill=blue!10] (C)--(BC)--(ABC)--(AC); 
		\draw [dashed] (E)--(B)--(BC); 
		\draw [dashed] (B)--(AB);
		
		
		\draw (E) node[vertex,label={below right:$\cst{e}{\emptyset}$}]{};
		\draw (A) node[vertex,label={below right: $\cst{e}{\{r\}}$}]{};		
		\draw (B) node[vertex,label={above right: $\cst{e}{\{s\}}$}]{};	
		\draw (C) node[vertex,label={above: $\cst{e}{\{t\}}$}]{};	
		\draw (D) node[vertex,label={above right: $\cst{e}{\{u\}}$}]{};	
		\draw (AB) node[vertex,label={below:$\cst{e}{\{r,s\}}$}]{};
		\draw (AC) node[vertex,label={above left: $\cst{e}{\{r,t\}}$}]{};
		\draw (BC) node[vertex,label={above: $\cst{e}{\{s,t\}}$}]{};	
		\draw (CD) node[vertex,label={above: $\cst{e}{\{t,u\}}$}]{};		
		\draw (ABC) node[vertex,label={above left: $\cst{e}{\{r,s,t\}}$}]{};										
		
		
		\draw [blue, thick,domain=0:180] plot ({1/3*cos(\x)},{ 1/3*sin(\x)},0);
		
		\filldraw[fill opacity=0.5,fill=blue,draw=blue] (0,1/3,0) arc (90:180:1/3)--(-1/3,0,0) arc (180:245:1/3 and .14)-- (0,0,1/3)arc (203:90: 0.14 and 1/3);

		\draw (.28,.27) node[label={[blue,shift={(0,0)}]$lk(\cst{e}{\emptyset})$}]{};

		\end{scope}

		\end{tikzpicture}
		\caption{ The fundamental domain  for the complex $\DGP$ where $\G$ is as shown with labels such that the $A_{\{r,s,t\}}$ is finite type. Note the 3-cube is filled in. The link of $\cst{e}{\emptyset}$ is shown in blue.} 
		
		\label{fig:Example link}
	\end{figure}

\begin{defn}\label{def:fundamentaldomain}
	We define the \emph{fundamental domain} of~$\DGP$ to be the finite subcomplex~$F_0$ consisting of all the cubes spanned by the cosets of the form $\cst{e}{T}$ for some $T\in \SF$.
\end{defn}

The entire complex $\DGP$ can be built by taking translates of this subcomplex by elements in the monoid and identifying vertices of $F_0$ with vertices of $\alpha F_0$ when the corresponding cosets are equal. It is for this reason that we call $F_0$ the fundamental domain of $\DGP$. See \Cref{fig:Example link} for an example of the fundamental domain.

	The lemma below shows that we can view the complex $\DGP$ as a subcomplex of $\DG$.

	\begin{lemma} \label{Lem:Injectivity of cosets} The map $\iota: \DGP \to \DG$ taking $\cst{\alpha}{T}$ to $\alpha A_T$ is injective and
		two vertices in $\DGP$ are connected by an edge if and only if their image in $\DG$ is connected by an edge.
	\end{lemma}
	
	\begin{proof}  For $T, R \in \SF$ and~$\alpha, \beta \in A_\G^+$,  it follows from Lemmas \ref{lem:cosets} and \ref{lem:group cosets} that $\cst{\alpha}{T} \subseteq \cst{\beta}{R}$ if and only if $\alpha A_T \subseteq \beta A_R$.  Two such cosets span an edge in $\DGP$ (respectively $\DG$) precisely when $|R \smallsetminus T|=1$. 
	\end{proof}

\section{Contractibility for arbitrary Artin groups} \label{Section:Contractibility}

In this section we prove that the monoid Deligne complex is contractible for \emph{any} Artin monoid. 

\begin{theorem}\label{thm:contactible}
	Let~$A_\G^+$ be an arbitrary Artin monoid. Then the cube complex~$\DGP$ is contractible.
\end{theorem}

Recall the definition of the fundamental domain~$F_0$ from Definition \ref{def:fundamentaldomain}. To show that $\DG^+$ is contractible, we construct it inductively, starting with the fundamental domain and adding translates of~$F_0$ by monoid elements of specific lengths. At each stage we prove contractibility. We make this inductive procedure precise below.

\begin{defn}
	Let~$\alpha\in A_\G^+$, and denote by~$\alpha F_0$ the translate of the fundamental domain~$F_0$ under left multiplication by~$\alpha$, that is,~$\alpha F_0$ consists of all the cubes spanned by cosets of the form $\cst{\alpha}{T}$ for some $T\in \SF$.
	For~$k\geq 0$, we define $D^+_k$ to be the union of the subcomplexes $\alpha F_0$ for all $\alpha$ with length $l(\alpha) \leq k$.
	\[
	D^+_k=\bigcup_{l(\alpha)\leq k} \alpha F_0.
	\]
\end{defn}

We state the following proposition, and prove Theorem~\ref{thm:contactible} assuming this proposition to be true. We then finish this section by proving the proposition.

\begin{proposition}\label{prop:unionofchambers}
	Let $\alpha \in \AG^+$ with $l(\alpha)=k$.  Then 
	\begin{enumerate}[(a)]
		\item \label{unionofchambsA} $\alpha F_0 \cap D^+_{k-1}$ is non-empty and contractible, and
		\item \label{unionofchambsB} for any $\beta\neq \alpha \in A^+_\G$ with $l(\beta)=k$, $\alpha F_0 \cap \beta F_0$ is contained in $D^+_{k-1}$.
	\end{enumerate}  
\end{proposition}

\begin{proof}[Proof of Theorem \ref{thm:contactible}]
	The proof is by induction on~$k$, noting that~$\DGP=\lim_{k\to \infty}D^+_k$. For the base case~$k=0$, we note that~$D_0^+=F_0$ so we must prove the fundamental domain is contractible. $F_0$ has a cone point, since~$\cst{e}{\emptyset}$ is a subset of every other coset in~$F_0$. Therefore~$F_0$ is contractible, and this proves the base case. 
	We assume, for our inductive hypothesis, that~$D^+_{k-1}\simeq \ast$. We consider~$D^+_k$ and show that, assuming Proposition~\ref{prop:unionofchambers}, this space is also contractible. Let~$\alpha \in A^+_\G$ satisfy~$l(\alpha)=k$. Then, using a similar argument as in the base case, the translate of the fundamental domain~$\alpha F_0$ is contractible. It follows from Proposition~\ref{prop:unionofchambers}~(\ref{unionofchambsA}) that~$D_{k-1}^+\cup \alpha F_0\cong \ast$ since this is the union of two contractible subcomplexes with (non-empty) contractible intersection. 
	
	Suppose~$\beta\neq \alpha \in A^+_\G$ with~$l(\beta)=k$.  Then  by Proposition~\ref{prop:unionofchambers}~(\ref{unionofchambsB}), the intersection~$\alpha F_0 \cap \beta F_0$ lies in~$D^+_{k-1}$, so  $(D_{k-1}^+\cup \alpha F_0) \cap \beta F_0= D_{k-1}^+ \cap \beta F_0$ which is contractible by part~(\ref{unionofchambsA}).  Thus as before, $(D_{k-1}^+\cup \alpha F_0)$ and $\beta F_0$ are two contractible spaces with contractible intersection, so ~$(D_{k-1}^+\cup \alpha F_0)\cup \beta F_0\cong \ast$.  Iterating this argument over all~$\gamma \in A^+_\G$ with~$l(\gamma)=k$, it follows that
	\[
	D^+_k=D^+_{k-1} \bigcup_{l(\gamma)=k} \gamma F_0
	\]
	is contractible, as required. This completes the induction, and thus the proof.
\end{proof}

\begin{proof}[Proof of Proposition \ref{prop:unionofchambers}]
	We first prove statement~(\ref{unionofchambsB}). 
	If~$\alpha F_0 \cap \beta F_0\neq\emptyset$ then there exist~$T_1$ and ~$T_2$ in~$\SF$ such that~$\cst{\alpha}{T_1}=\cst{\beta}{T_2}$. Applying Lemma~$\ref{lem:cosets}$ in both directions gives that~$T_1=T_2=T$. It is shown in Lemma~\ref{lem:representative} that any coset $\cst{\alpha}{T}$ has a unique shortest element, namely $\bar\alpha_T$.  So if $\cst{\alpha}{T} = \cst{\beta}{T}$ and $l(\alpha)=l(\beta)=k$, then either $\alpha = \beta$ (a contradiction) or the length of $\bar\alpha_T$ is strictly less than~$k$. In the latter case it follows that~$\cst{\alpha}{T}=\cst{\bar\alpha}{T}\in D^+_{k-1}$.
	
	For statement~(\ref{unionofchambsA}), note that a vertex corresponding to the coset $\cst{\alpha}{T}$ lies in $\alpha F_0 \cap D^+_{k-1}$ if and only if $\bar\alpha_T$ has length less that $k$, in which case $\alpha=\bar\alpha_T \rho$ for some non-trivial element $\rho \in A^+_T$.  In particular, for some~$t\in T$, $\alpha \sgeq_R t$. 
	Recall that $T_\alpha$ denotes the subset of generators~$s$ such that $\alpha\sgeq_R s$, and by Lemma~\ref{lem:endgenfinitetype}, $T_\alpha\in \SF$.  The observation above can now be stated as follows:  $\cst{\alpha}{T}$ lies in $\alpha F_0 \cap D^+_{k-1}$ if and only if $T \cap T_\alpha \neq \emptyset$.  In particular $\alpha F_0 \cap D^+_{k-1}$ is non-empty, since~$\alpha\neq e$ and so for at least one~$s\in S$,~$s\in T_\alpha$ and it follows that~$\cst{\alpha}{s} \in D^+_{k-1}$.
	
	Let $Y$  be the subcomplex of~$\alpha F_0$ spanned by the cosets $\cst{\alpha}{T}$ such that $T \cap T_\alpha \neq \emptyset$. Then statement~(\ref{unionofchambsA}) is equivalent to showing that~$Y$ is contractible.  Let $Y_0$ be the subcomplex of~$Y$ spanned by $\cst{\alpha}{T}$ such that $T \subseteq T_\alpha$.  Then $Y_0$ is contractible since it contains a maximal element $\cst{\alpha}{T_{\alpha}}$.  Define a projection map 
	\[ p\colon Y \to Y_0, \,\, \cst{\alpha}{R} \mapsto \cst{\alpha}{R \cap T_{\alpha}}.\]  
	We claim that $p$ is a deformation retraction and hence $Y$ is also contractible.  To see this, let $\mathcal R$ denote the poset of subsets $R\in \SF$ with $R \cap T_\alpha \neq \emptyset$.  Then cubes in $Y$ are in one-to-one correspondence with intervals $I=[R_1,R_2]$ in $\mathcal R$.  This cube, together with its projection $p(I)=[R_1 \cap T_\alpha, R_2 \cap T_\alpha]$ spans a larger cube  $C_I= [R_1 \cap T_\alpha, R_2]$, since~$R_1\subset R_2$ implies~$R_1\cap T_\alpha \subset R_2$.   Then the restriction of $p$ to the cube $C_I$ is a deformation retraction of $C_I$ onto $C_I \cap Y_0$ (see \Cref{fig:DeformRetract}).  Moreover, if $I'=[R_2,R_3] \subset I$ is a subinterval, and $C_{I'} = [R_2 \cap T_\alpha, R_3]$ is the corresponding face of $C_I$, then $p|_{C_I}$ restricts to the corresponding deformation retraction of $C_{I'}$ onto $C_{I'} \cap Y_0$, i.e. $(p|_{C_I})|_{C_{I'}}=p|_{C_{I'}}$.  It follows that the restrictions of~$p$ to all such cubes~$C_I$ for~$I\in \mathcal{T}$ glue together along common faces to give the desired deformation retraction of $Y$ onto $Y_0$. 	
\end{proof} 

	\begin{figure}
		\tdplotsetmaincoords{70}{90}
		\tdplotsetrotatedcoords{60}{0}{60}
		\begin{tikzpicture}[scale=2.5, tdplot_rotated_coords]
		
		

		\coordinate(E) at (0,0,0);
		\coordinate(N) at (0,0,1);
		\coordinate(R1) at (0,1,1);
		\coordinate(R2) at (1,0,1);
		\coordinate(RA) at (1,1,1);
		\coordinate(T1) at (0,1,0);
		\coordinate(T2) at (1,0,0);
		\coordinate(TA) at (1,1,0);
		
		\draw (E)-- (T2)--(TA)--(T1)-- cycle;
		\filldraw[fill=green!10] (T2)-- (R2)--(RA)--(TA)-- cycle;
		\filldraw[fill=green!10] (R1)-- (RA)--(TA)--(T1)-- cycle;
		\draw[very thick, blue] (T2)--(TA)--(T1);

		\draw (E)  node[vertex,label={below left:$\cst{\alpha}{\emptyset}$}]{};
		\draw (R1) node[vertex,green!50!black,label={above left:$\cst{\alpha}{R_1}$}]{};
		\draw (RA) node[vertex,green!50!black,label={above right:$\cst{\alpha}{R_3}$}]{};
		\draw (R2) node[vertex,green!50!black,label={above right:$\cst{\alpha}{R_2}$}]{};
		\draw (T1) node[vertex, blue, label={below left:$\cst{\alpha}{R_1\cap T_\alpha}$}]{};
		\draw (T2) node[vertex,blue,label={below right:$\cst{\alpha}{R_2\cap T_\alpha}$}]{};
		\draw (TA) node[vertex,blue,label={below:$\cst{\alpha}{T_\alpha}$}]{};
		
		\draw[->,thick, dashed, green!50!black] (0.3,1,1)--(0.3,1,0.53);
		\draw[->,thick, dashed, green!50!black] (0.3,1,0.47)--(0.3,1,0);
		\draw[->,thick, dashed, green!50!black] (0.6,1,1)--(0.6,1,0.53);
		\draw[->,thick, dashed, green!50!black] (0.6,1,0.47)--(0.6,1,0.01);
		
		\draw[->,thick, dashed, green!50!black] (1,0.3,1)--(1,0.3,0.53);
		\draw[->,thick, dashed, green!50!black] (1,0.3,0.47)--(1,0.3,0.02);
		\draw[->,thick, dashed, green!50!black] (1,0.6,1)--(1,0.6,0.53);
		\draw[->,thick, dashed, green!50!black] (1,0.6,0.47)--(1,0.6,0.03);
		
		\end{tikzpicture}
		
		\caption{An example of the deformation retraction described in the proof of \Cref{prop:unionofchambers}. In this example we assume that $T_\alpha\subset R_3$. The subcomplex $Y$, shown in green, first retracts along the arrows to the subset $Y_0$ in blue, and then the $Y_0$ retracts to the cone point $\cst{\alpha}{T_\alpha}$. 
		}
		\label{fig:DeformRetract}
	\end{figure}
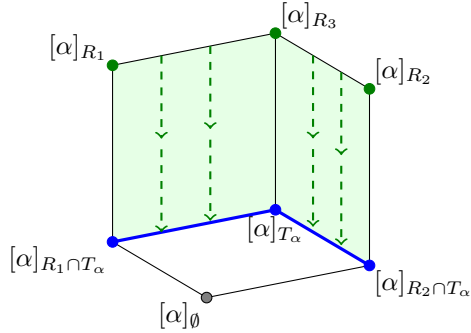 

\section{Convexity and CAT(0) in FC-type case} \label{Section:Convexity}

	In this section we will look at the geometric relationship between the monoid Deligne complex, $\DGP$ and the Deligne complex, $\DG$. By \Cref{Lem:Injectivity of cosets}, we may view $\DGP$ as a subcomplex of $\DG$.
	
	Our main goal in this section is to prove the following theorem.
	
	\begin{theorem}  \label{Thrm:Local Convexity}  The inclusion map $\iota: \DGP \to \DG$ is a locally isometric embedding.  
	\end{theorem}
	
	Before proving the theorem, we discuss some consequences.  While the theorem holds for all Artin groups, it has many additional implications for FC-type Artin groups since in that case, $\DG$ is CAT(0).
	
	\begin{corollary}  \label{Cor:DGPisCAT(0)}
		If $\AG$ is an FC-type Artin group, then the inclusion map $\iota: \DGP \to \DG$ is an isometric embedding, hence $\DGP$ is CAT(0) and its image is convex in $\DG$.
	\end{corollary}
	
	Here is a proof of the corollary assuming \Cref{Thrm:Local Convexity}
	\begin{proof}  If $\AG$ is FC-type, then $\DG$ is CAT(0).
		In a CAT(0) space local geodesics are globally geodesic, thus a local isometry $\iota$ takes a geodesic between $x$ and $y$ to a geodesic between $\iota(x)$ and $\iota(y)$.  Since distance is measured by the length of geodesics, it follows that $\iota$ is a (globally) isometric embedding.  Moreover, since geodesics in a CAT(0) space are unique, this implies that the geodesic in $\DG$ between two points in the image of $\iota$ also lies in the image of $\iota$.  That is, $\iota(\DGP)$ is convex in $\DG$.  The CAT(0) condition is inherited by any convex subspace, so we conclude that $\DGP$ is CAT(0).
	\end{proof}
	
	Another important consequence of the above corollary is that if $\AG$ is FC-type, then for any two vertices $x,y$ in $\DGP$, their cubical convex hull lies entirely in $\DGP$,  that is, any minimal length edge path in $\DG$ between $x$ and $y$ remains inside $\DGP$. 
	
	Next consider the action of $A_\G^+$ on $\DGP$.  For any cube complex $X$, a continuous map $g: X \to X$ that takes cubes isometrically to cubes, is distance non-increasing.  In the case of a group action, this is sufficient to show that $G$ acts by isometries since the inverse map $g^{-1}$ is also distance non-increasing.  In the case of a monoid action, this need not be true.  The map $g$ need not be surjective and it may decrease distances.  Indeed, the action of a non-trivial element $\alpha \in A_\G^+$ on $\DGP$ is never surjective (in particular, the image does not contain $\cst{e}{\emptyset}$).  Another consequence of  \Cref{Cor:DGPisCAT(0)}, however,
	is that the action is distance preserving.
	
	\begin{corollary}
		Suppose that $\AG$ is an FC-type Artin group. Then action of $\alpha\in \AG^+$ induces an isometric embedding $\DGP \to \DGP$. 
	\end{corollary} 
	
	\begin{proof}
		The action of $\alpha$ on $\DG$ induces an isometry, so this fact, combined with \Cref{Cor:DGPisCAT(0)}, shows that translation by $\alpha$ preserve distances between points. 
	\end{proof}

	\subsection{Proof of \ref{Thrm:Local Convexity} }
	
	To prove \Cref{Thrm:Local Convexity}, we will apply the following lemma.  A proof of the lemma can be found in \cite{Haglund2008}, but we include an outline of the proof below for the completeness. 
	
	\begin{definition}  A subcomplex $K$ of a simplicial complex $L$ is said to be a \emph{full subcomplex} if any collection of vertices in $K$ that spans a simplex in $L$, also spans a simplex in $K$.
	\end{definition}
	
	\begin{lemma}\label{Lem:LocalEmbedding}  Let $X$ be a cube complex and $Y \subseteq X$ a subcomplex.  
		Suppose that for every vertex $v \in Y$, the link of $v$ in $Y$ is  full subcomplex of the link of $v$ in $X$.  Then the inclusion map $Y \to X$ is a locally isometric embedding, where the metrics are given by minimal path lengths in~$Y$, respectively~$X$.
	\end{lemma}

	\begin{proof} Here is  sketch of the proof.  To be a local geodesic in a cube complex, a path $p$ must first be piecewise linear, that is $p=p_1p_2 \dots p_k$ where each $p_i$ is a straight line lying in a single cube.  In addition, at the point $x_i$ where $p_i$ meets $p_{i+1}$, the tangent vectors to these two segments must have distance at least $\pi$ in the link of $x_i$.    Thus to show that any path that is locally geodesic in $Y$ is also locally geodesic in $X$, we must show that for any point $y \in Y$, two points in $lk_Y(y)$ of distance $\geq \pi$, are also of distance $\geq \pi$ in $lk_X(y)$.  
		
		By standard arguments, one can reduce to checking the case where $y=v$ is a vertex in $Y$.  Let $\ell_1, \ell_2$ be two points in $lk_Y(v) \subseteq lk_X(v)$.  The distance between them measured in $lk_X(v)$, is the minimal length of a path $\gamma$ connecting these two points.  If such a path $\gamma$ lies entirely in $lk_Y(v)$, then the distance from $\ell_1$ to $\ell_2$  in $lk_Y(v)$ is equal to their distance in $lk_X(v)$. If $\gamma$ exits $lk_Y(v)$, then the fact that $lk_Y(v)$ is a full subcomplex of $lk_X(v)$, means that it enters a simplex containing a vertex $w$ in $lk_X(v) \smallsetminus lk_Y(v)$.  Let $\gamma'$ be a maximal segment of $\gamma$ whose interior lies in the open star of $w$.  The subspace of the star of $w$ spanned by $w$ and $\gamma'$ can be identified with a subspace of the 2-sphere, with $w$ as the north pole and $\gamma'$ a geodesic in the upper hemisphere with endpoints on the equator.  Any such geodesic has length $\pi$,
		thus the distance between $\ell_1$ and $\ell_2$ in $lk_X(v)$ is~$\geq \pi$.
	\end{proof}
	
	In light of this lemma, to prove \Cref{Thrm:Local Convexity}, it remains to show  that for any vertex $v$ in $\DGP$, the link of $v$ in $\DGP$ (denoted $\lkp{v}$) is a full subcomplex of the link of $v$ in $\DG$ ( denoted $\lk{v}$).
	We approach this problem by splitting the link of a vertex into two pieces, the upward link and the downward link, such that the link of $v$ is the join of the upward and downward links. 
	
	\begin{defn}
		If $v$ is a vertex corresponding to the monoid coset $\cst{\alpha}{T}$, then any vertex in $\lkp{v}$ corresponds to a coset which is either included in or contains $\cst{\alpha}{T}$. The vertices in $\lkp{v}$ can be partitioned into two sets according to the direction of this inclusion and we call the subcomplexes spanned by these sets the \emph{upward and downward links} of $v$ in $\DGP$.  
		We define upward and downward links of vertices $\alpha A_T$ in $\DG$ similarly.  
	\end{defn} 
	
	By \Cref{Lem:Injectivity of cosets}, the upward link in $\lkp{v}$ is a subcomplex of the upward link in $\lk{v}$ and the downward link in $\lkp{v}$ is a subcomplex of the downward link in $\lk{v}$.
	
	We now focus on the upward and downward links in turn. 
	
	\begin{lemma} \label{UpwardLink}
		Let $\AG$ be an Artin group with monoid $\AG^+.$ For any vertex $v= \cst{\alpha}{T}\in \DGP$, the map $\iota$ sends the upward link of $\cst{\alpha}{T}$ in $\DGP$ to a full subcomplex of the upward link  of $\alpha A_T$ in $\DG$. 
	\end{lemma}
	
	\begin{proof}
		By \Cref{lem:cosets} the vertices in the upward link can each be written as $\cst{\alpha}{R}$ where $R\in\SF$ and  $R=T\cup r \text{ for some } r\in S$.
		Now suppose that $\cst{\alpha}{T_1}\dots \cst{\alpha}{T_n}$ is a collection of vertices in the upward link, where each $T_i=T\cup s_i$ for some $s_i$ in the generating set $S$. Suppose further that the images of these vertices under $\iota$, $\alpha A_{T_i}$, span a simplex in $\DG$. This implies that $T'=\cup_i(T_i)$ is in $\SF$. 
		Thus $\cst{\alpha}{T'}$ is a vertex in the complex $\DGP$, and the vertices  $\cst{\alpha}{T_i}$ span a simplex in the link of $ \cst{\alpha}{T}$ in $\DGP$.
	\end{proof}
	
	Now we address the more difficult case of the downward link. The proof that downward links in $\DGP$ are mapped to full subcomplexes of downward links in $\DG$ will involve several steps, starting with the following lemma.
	
	\begin{lemma}\label{OneSkeleton} Let $\AG$ be an Artin group with monoid $\AG^+.$ For any vertex $v=\cst{\alpha}{T}\in \DGP$, the map $\iota$ sends the one-skeleton of the downward link of $\cst{\alpha}{T}$ to a full subgraph of the one-skeleton of the downward link of $\alpha A_{T}$.  
	\end{lemma}
	
	\begin{proof}  Suppose $\cst{a_1}{T_1}$ and $\cst{a_2}{T_2}$ are vertices in the downward link of $\cst{\alpha}{T}$. 
		
		Assume $\alpha$ is the minimal representative in $\cst{\alpha}{T}$.  Then left multiplication by $\alpha$ preserves the inclusion relation on cosets and  maps the downward link of $\cst{e}{T}$ (where $e$=identity) isomorphically to the downward link of $\cst{\alpha}{T}$. Thus we may assume without loss of generality that $\alpha=e$, and $a_1,a_2 \in A_T^+$.

		Let $T_{12}=T_1\cap T_2$.
		By assumption, $a_1A_{T_1}, a_2A_{T_2}, A_T$ lie in a cube in $\DG$ spanned by $cA_{T_{12}}$ and  $A_T$  for some $c \in a_1A_{T_1} \cap a_2A_{T_2}$.  By Lemma \ref{lem:garsideform}, any element of $A_{T_1}$ can be written in the form $b_1\Delta_1^{-k}$ where $b_1 \in A_{T_1}^+, k \geq 0$ and $\Delta_1$ is the Garside element for $A_{T_1}$.  Thus we can write $c=a_1b_1\Delta_1^{-k}$  and likewise $c=a_2b_2\Delta_2^{-j}$.  
		Say $k \geq j$.  Then replacing $b_2$ by $b_2\Delta_2^{k-j}$, we may assume that $k=j$, that is,
		$$c = a_1b_1 \Delta_1^{-k} = a_2b_2 \Delta_2^{-k}.$$

		If $k=0$, then $c$ lies in the monoid $A_T^+$.  Hence the interval $[\cst{c}{T_{12}}, \cst{e}{T}]$ spans a cube in $\DGP$. So suppose $k > 0$.  Let $d = \gcd_L(\Delta_1^{k}, \Delta_2^{k})$ be the maximal left divisor, and write $\Delta_i^{k} = d z_i$.  
		We will show that $cd \in a_1A_{T_1}^+ \cap a_2A_{T_2}^+$. First, note that $cd=a_ib_iz_i^{-1} \in a_iA_{T_i}$ for $i=1,2$, so it remains only to check that $cd$ lies in the monoid. For this, note  that $(a_1b_1)^{-1}(a_2b_2) = z_1^{-1}z_2$.  Since
		$\gcd_L(z_1, z_2) = e$,  $z_1^{-1}z_2$ is the unique minimal representative for this element (unique by~\cite[Theorem 2.6]{Charney1995a}). Let~$g=\gcd_L(a_1b_1, a_2b_2)$, so there exists some~$y_1$ and~$y_2$ in~$\AG^+$ with $\gcd_L(y_1, y_2) = e$ such that~$a_ib_i=gy_i$. Then 
		\[
		(a_1b_1)^{-1}(a_2b_2)=y_1^{-1}g^{-1}gy_2=y_1^{-1}y_2=z_1^{-1}z_2
		\]  
and by uniqueness it follows that~$y_i=z_i$. Therefore $a_ib_i \sgeq_R z_i$, and it follows that $cd=a_ib_iz_i^{-1}$ is in $a_iA_{T_i}^+$.  
	\end{proof}

	Now we turn to the case of higher dimensional simplices in the downward link. 
	
	\begin{lemma} \label{DownwardLinkFlag}
		Let $\AG$ be an Artin group with monoid $\AG^+.$ For any vertex $v\in \DGP$, the downward link of this vertex is flag.  
	\end{lemma}
	
	\begin{proof}
		As before, left multiplication by $\alpha$ preserves the partially ordered set on the cosets. The downward link of $v=\cst{\alpha}{T}$ is given by a copy of the Deligne monoid complex for $A^+_T$, left multiplied by $\mr{\alpha}{T}$.
		So it is sufficient to show this result in that case of a vertex $\cst{e}{T}$, where $e$ is the identity element.  
		
		Vertices in the downward link of $\cst{e}{T}$ are of the form $\cst{a_i}{T_i}$, where $a_i\in A_T^+$ and  $T_i=T\backslash \{t_i\}$, for some $t_i\in T$. Suppose we have a set of vertices in this form for $1\leq i\leq n$ and each pair of vertices in this set spans an edge in the downward link. In other words, suppose that  $\cst{a_i}{T_i} \cap \cst{a_j}{T_j}\neq \emptyset$ for all pairs $\{i,j\}$. We would like to show that these vertices span an $n$-simplex in the downward link by showing that $\cap_{i} \cst{a_i}{T_i}\neq \emptyset$.

		First we will find an expression for $\cst{a_i}{T_i} \cap \cst{a_j}{T_j}$ as a single coset. Let $T_{ij}=T_i\cap T_j$. The fact that $\cst{a_i}{T_i} \cap \cst{a_j}{T_j}\neq \emptyset$ and  $\cst{a_i}{T_i} \neq \cst{a_j}{T_j}$ implies that $T_i\neq T_j$ and $T_{ij}$ is a strict subset of these sets. This means that if $\beta\in A_T^+$ is in the intersection $\cst{a_i}{T_i} \cap \cst{a_j}{T_j}$, then this intersection can be written as $\cst{\beta}{T_{ij}}$. However we would like a more precise expression for $\beta$. 
		
		\textbf{Claim}:~$\cst{a_i}{T_i} \cap \cst{a_j}{T_j}\neq \emptyset \iff$  $\llm(a_i,a_j)$ exists and  
		$\cst{a_i}{T_i} \cap \cst{a_j}{T_j}=\cst{\llm(a_i,a_j)}{T_{ij}}.$
		
		{Proof of claim}:~($\Leftarrow$) is immediate.\\
		($\Rightarrow$) Suppose~$\cst{a_i}{T_i} \cap \cst{a_j}{T_j}\neq \emptyset$. Then there exists a common multiple~$x$ such that~$x=a_im_i$ and~$x=a_jm_j$ for~$m_i\in A^+_{T_i}$ and~$m_j\in A^+_{T_j}$.  So $\llm(a_i,a_j)$ exists and~$x=\llm(a_i,a_j)m$ for some $m$. Write~$\llm(a_i,a_j)=a_ib_i=a_jb_j$ and compare~$x=\llm(a_i,a_j)m=a_ib_im=a_jb_jm$ to~$x=a_im_i$ and~$x=a_jm_j$. By cancellation of~$a_i$ and~$a_j$ on the left it follows that~$b_im\in A^+_{T_i}$ and~$b_jm \in A^+_{T_j}$. 
		Therefore $m\in A^+_{T_{ij}}$
		and~$x\in \cst{\llm(a_i, a_j)}{T_{ij}}$.
		This shows~$\cst{a_i}{T_i} \cap \cst{a_j}{T_j} \subseteq \cst{\llm(a_i,a_j)}{T_{ij}}.$
		
		To show~$\supseteq$, note that from above~$b_i\in A^+_{T_i}$ and~$b_j \in A^+_{T_j}$. Therefore if~$y=\llm(a_i,a_j)m$ for~$m\in A^+_{T_{ij}}$ then it follows that~$y=a_ib_im=a_jb_jm$ and since~$m\in A^+_{T_{ij}}\subset A^+_{T_i} \text{ and }A^+_{T_j}$, then~$y \in \cst{a_i}{T_i} \cap \cst{a_j}{T_j}$.
		
		

		Now we turn our attention to the intersection $\cap_{i} \cst{a_i}{T_i}$.

		\textbf{Claim}: if~$\cst{a_i}{T_i} \cap \cst{a_j}{T_j}=\cst{\llm(a_i,a_j)}{T_{ij}}$ is non-empty for all pairs $\{i,j\}$, then $\cap_{i} \cst{a_i}{T_i}$ is also non-empty.
		
		{Proof of claim}: Set~$i=1$. We have that~$\llm(a_1,a_j)$ is in $\cst{a_1}{T_1} \cap \cst{a_j}{T_j}$ for all $j$ and so we can write $\llm(a_1,a_j)=a_1 m_j=a_jn_j$ for~$m_j\in A^+_{T_1}$ and~$n_j \in A^+_{T_j}$. 
		
		Since, for every $2\leq j\leq n$, ~$m_j$ is in~$A^+_{T_1}$, and~$A^+_{T_1}$ is finite type (by definition), it follows that~$\llm(\{m_j\})$ exists and is in~$A^+_{T_1}$. So~$a_1\llm(\{m_j\})\in \cst{a_1}{T_1}$ satisfies that~$a_i$ is a left divisor of this element for all $i$. Since they have a common multiple, they have a least common multiple~$\llm(\{a_i\})$ for $1\leq i\leq n$.

		We now show that $\llm(\{a_i\})$ is in~$\cst{a_1}{T_1}$.  Suppose it is not, then~$\llm(\{a_i\})=a_1 x$ for~$x\in A_T^+$ but $x\notin A^+_{T_1}$. Since~$a_1\llm(\{m_j\})$ is a common multiple, the least common multiple is a left divisor of it, so~$a_1\llm(\{m_j\})=a_1 x y$ for some~$y\in A^+$. Then cancellation of~$a_1$ gives~$xy=\llm(\{m_j\}) \in A^+_{T_1}$. This contradicts~$x\notin A^+_{T_1}$. 
		
		A similar argument shows that $\llm(\{a_i\})$ is in~$\cst{a_k}{T_k}$ for any $1\leq k\leq n$. This shows that   $\llm(\{a_i\})$ is in the intersection $\cap_{i} \cst{a_i}{T_i}$ completing the proof. 
	\end{proof}

	Since flag complexes are completely determined by their one-skeleton, combining \Cref{OneSkeleton,DownwardLinkFlag}, gives the desired result on downward links.

	\begin{lemma} \label{DownwardLink}
		Let $\AG$ be an Artin group with monoid $\AG^+.$ For any vertex $v= \cst{\alpha}{T}\in \DGP$, the map $\iota$ sends the downward link of $\cst{\alpha}{T}$ in $\DGP$ to a full subcomplex of the downward link  of $\alpha A_T$ in $\DG$. 
	\end{lemma}

	Finally we show that the link of a vertex in $\DGP$ is the join of the upward and downward links. 
	
	\begin{lemma} \label{FullSubcomplex}
		Suppose that $\AG$ is an arbitrary Artin group and $v=\cst{\alpha}{T}$ is a vertex in $\DGP$. Then the map $\iota$ takes $\lkp{v}$ to a full subcomplex of $\lk{\iota(v)}$.
	\end{lemma}
	
	\begin{proof}
		
		Suppose that $X$ is a set of vertices in $\lkp{v}$, and suppose that these vertices span a simplex in $\lk{v}$. The set $X$ can be partitioned into two sets $X_u$ and $X_d$ in the upward and downward links respectively. By \Cref{UpwardLink,DownwardLink} the sets $X_u$ and $X_d$ must span simplices in $\lkp{v}$. If either $X_u$ or $X_d$ is empty then we are done. 
		
		Assume  $X_u$ and $X_d$  are non-empty.  Simplices in $\lkp{v}$ correspond to cubes in $\DGP$ containing $v$, or equivalently, intervals containing $v$.  The simplex spanned by $X_u$ corresponds to an interval $[v,w]$ for some $w \in X_u$ while the simplex spanned by $X_d$ corresponds to an interval $[z,v]$ for some $z \in X_d$.  It follows that $[z,w]$ is an interval containing $v$ and the corresponding simplex in $\lkp{v}$ is precisely the span of $X$. 
	\end{proof} 
	
	

	Combining  \Cref{Lem:LocalEmbedding} with \Cref{FullSubcomplex}, this completes the proof of \Cref{Thrm:Local Convexity}.

\section{Properties of the monoid embedding} \label{Section:CayleyGraph}

	Our focus so far has been on the monoid Deligne complex and its relation to the full Deligne complex.  In this section we consider the relation between monoid Cayley graphs and the group Cayley graphs.  
	Let $S$ be the standard generating set for an Artin group $\AG$ and let ~$\calG(\AG,S)$ denote the corresponding Cayley graph.  As noted above, by a result of Paris~\cite{Paris2002} the Artin monoid $\AG^+$ injects into the Artin group so we can identify elements of $\AG^+$ with vertices in this Cayley graph.  
	
	\begin{defn}
		The \emph{Artin monoid Cayley graph}, $\calG^+(\AG, S)$, is  the full subgraph of~$\calG(\AG,S)$ spanned by the vertices~$v\in \AG^+$.  
	\end{defn}
	
	Note that when considering a path in the Artin monoid Cayley graph, one can traverse either forwards or backwards along edges, i.e.~between monoid elements~$v$ and~$v\cdot s$ for~$s\in \S$.
	
	Consider the induced metric on $\calG^+(\AG, S)$.  That is, the distance between two vertices
	$a,b \in \AG^+$ is the length of the shortest path in $\calG^+(\AG, S)$ connecting them.    It is interesting to ask whether $\calG^+(\AG, S)$ embeds isometrically (respectively convexly) in  $\calG(\AG,S)$, that is, whether some (respectively any) minimal length path from $a$ to $b$ in $\calG^+(\AG, S)$ is also minimal length in $\calG(\AG,S)$.  {Since translation by $a^{-1}$ does not preserve $\calG^+(\AG, S)$, the problem does not simply reduce to the case where $a=e$, and it seems quite subtle in general.}
	
	We believe that a more promising approach is to use a slightly larger generating set, namely the set of minimal elements as defined below.
	
	\begin{defn}
		Let~$\calM$ be the set of all \emph{minimal} elements in~$\AG$. That is
		\[
		\calM=\{m\in \AG^+ \mid m\sleq_L \Delta_T, \, T\in \SF\}.
		\]
	\end{defn}
	
	The minimal elements in a finite type subgroup $A_T$ are in one-to-one correspondence with the non-trivial elements in the Coxeter group $W_T$.  
	For the remainder of this section we denote the corresponding Cayley graphs by  $\calG=\calG(\AG, \calM)$ and $\calG^+=\calG^+(\AG, \calM)$. 
	
	In the case of a finite type Artin group $\AG$, one can algorithmically find a normal form for elements $g \in \AG$ which is geodesic with respect to this generating set~\cite{Charney1995a}.  This normal form is obtained by first factoring $g$ into a product $g=a^{-1}b$, where $a,b \in \AG^+$ and $\gcd_L(a,b)=e$, and then factoring each of $a$ and $b$ into a product of minimals in a canonical way
	(called the right greedy normal form of $a$ and $b$.)  
	
	\begin{proposition} \label{Prop:IsoEmdbedding}
		Suppose~$\AG$ is a finite type Artin group, then $\calG^+(\AG, \calM)$ embeds isometrically in $\calG(\AG, \calM)$.
	\end{proposition}
	
	\begin{proof}
		Let $a,b$ be elements of $\AG^+$.  Then a minimal length edge path in $\G$ from $a$ to $b$ corresponds to a minimal length word in the generating set $\calM$ representing the group element $a^{-1}b$.  We claim that at least one such  path lies entirely in $\AG^+$.  
		
		To see this, let $c = \gcd_L(a,b)$ and write $a=ca', b=cb'$.  Let  $a'=\mu_1\mu_2 \dots \mu_k$ and $b'= \eta_1 \eta_2 \dots \eta_j$ be  the right greedy normal forms.  Then by \cite{Charney1995a}, the word $w=\mu_k^{-1} \dots \mu_1^{-1} \eta_1 \dots \eta_j$ is a minimal length representative for $a^{-1}b = (a')^{-1}(b')$.
		
		Now consider the paths $\gamma_1$ from $e$ to $a'$ and $\gamma_2$ from $e$ to $b'$ given by the words $\mu_1 \dots \mu_k$ and $\eta_1 \dots \eta_j$ respectively.  Both of these paths lie entirely in $\calG^+$ and the path that traverses $\gamma_1$ in reverse followed by $\gamma_2$ is a minimal length edge path from $a'$ to $b'$ in both $\calG^+$ and $\calG$.  Translating this path by $c$ gives a minimal length edge path from $a$ to $b$ which again lies entirely in $\calG^+$.  
	\end{proof}
	
	Recall that a subspace $Y$ of a geodesic metric space $X$ is said to be convex if every geodesic in $X$ with endpoints in $Y$ lies entirely in $Y$.  It is said to be quasi-convex if there exists an $N > 0$ such that every geodesic in $X$ with endpoints in $Y$ lies in the $N$-neighbourhood of $Y$.
	In light of the \Cref{Prop:IsoEmdbedding}, it is reasonable to ask whether $\calG^+$ is convex or at least quasi-convex, in $\calG$.   This turns out to be false in general as the following example demonstrates.  
	
	\begin{exam}
		Let $\AG$ be an Artin group whose defining graph contains a subgraph $\G'$ of the following form (here~$m$ can be any label).  For example, this holds for any braid group with at least four generators.

		
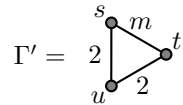
\begin{center}
			\begin{tikzpicture}[thick, scale=.7]
			
			\begin{scope}[scale=1.2,shift={(-5,3)}]
			\draw (-2,0) node[]{{$\G'=$}};
			
			\draw[rotate =150] (0,0)--(1,0)--(1/2,0.87)--(0,0);
			
			\draw (0,0) node[vertex,label={above right: $t$}]{};
			\draw[rotate =150] (1,0) node[vertex,label={above left:$s$}]{};
			\draw[rotate =210] (1,0) node[vertex,label={below left: $u$}]{};
			\draw (-1,0) node[left] {$2$};
			\draw (-.25,-.5) node[left] {$2$};
			\draw (-.2,0.5) node[left] {$m$};
			\end{scope}
			\end{tikzpicture}
		\end{center}
		
		Consider the Cayley graph $\calG$ of $\AG$ with respect to the set of minimal elements 
		and the corresponding monoid Cayley graph $\calG^+$.
		We claim that $\calG^+$ is not quasi-convex in $\calG$.  To see this, 	
		we will show that for every~$k\in \N$, there exist  elements~$a,b$ in~$\AG^+$ with a geodesic~$\gamma$ between them lying entirely in~$\calG^+$ and a geodesic $\gamma'$ in~$\calG$ which travels outside the~$k$ neighbourhood of~$\gamma$.
		
		Fix~$k\in \N$ and let~$n=k+1$. Consider the elements~$a=s^n$ and~$b=t^n$. Then the geodesic~$\gamma$ pictured below lies entirely in~$\calG^+$.
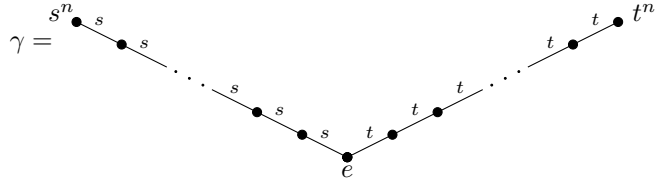
		\begin{center}
			\begin{tikzpicture}[scale=0.6]
			\foreach \x in {0,1,2,5,6}
			\draw[fill=black] (\x, \x*.5) circle [radius=0.1](-\x, \x*.5) circle [radius=0.1];
			\foreach \x in {3.2,3.5,3.8}
			\draw[fill=black] (\x, \x*.5) circle [radius=0.02](-\x, \x*.5) circle [radius=0.02];
			\draw (-3,1.5)--(0,0)--(3,1.5) (4,2)--(6,3) (-4,2)--(-6,3);
			\foreach \x in {.5,1.5,2.5,4.5,5.5}
			\draw (\x,\x*.5) node[above,yshift=.1cm] {$\scriptstyle{t}$} (-\x, \x*.5) node[above,yshift=.1cm] {$\scriptstyle{s}$};
			\draw (-6,3) node[above, xshift=-.2cm] {$s^n$} (6,3) node[above,xshift=.35cm] {$t^n$};		
			\draw (-7,2.5) node {$\gamma =$};
			\draw (0,0) node[below, yshift=-.1cm] {$e$};
			\end{tikzpicture}
		\end{center}
		However, since~$\Delta_{s,u}=su$ and ~$\Delta_{t,u}=tu$,~$su, tu \in \calM$. It follows that the path~$\gamma'$ below is also of length~$2n$, and hence a geodesic in~$\calG$. Since~$n=k+1$, this geodesic does not lie in a~$k$-neighbourhood of~$n$. thus, we have constructed the required example.
		
		
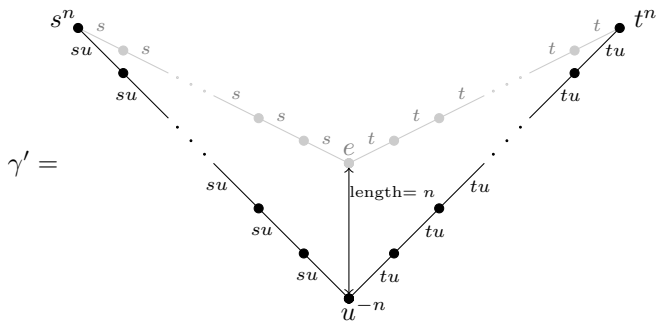
\begin{center}
			\begin{tikzpicture}[scale=0.6]
			\foreach \x in {0,1,2,5,6}
			\filldraw[color=gray!40] (\x, \x*.5) circle [radius=0.1](-\x, \x*.5) circle [radius=0.1];
			\foreach \x in {3.2,3.5,3.8}
			\draw[gray!40] (\x, \x*.5) circle [radius=0.02](-\x, \x*.5) circle [radius=0.02];
			\draw[gray!40] (-3,1.5)--(0,0)--(3,1.5) (4,2)--(6,3) (-4,2)--(-6,3);
			\foreach \x in {.5,1.5,2.5,4.5,5.5}
			\draw[gray] (\x,\x*.5) node[above,yshift=.1cm] {$\scriptstyle{t}$} (-\x, \x*.5) node[above,yshift=.1cm] {$\scriptstyle{s}$};
			\draw (-6,3) node[above, xshift=-.2cm] {$s^n$} (6,3) node[above, xshift=.35cm] {$t^n$};			
			\draw[gray] (0,0) node[above, yshift=.1cm] {$e$};

			\foreach \x in {0,1,2,5,6}
			\draw[fill=black] (\x, \x-3) circle [radius=0.1](-\x, \x-3) circle [radius=0.1];
			\foreach \x in {3.2,3.5,3.8}
			\draw[fill=black] (\x, \x-3) circle [radius=0.02](-\x, \x-3) circle [radius=0.02];
			\draw (-3,0)--(0,-3)--(3,0) (4,1)--(6,3) (-4,1)--(-6,3);
			\foreach \x in {.5,1.5,2.5,4.5,5.5}
			\draw (\x,\x-3) node[right,xshift=.1cm] {$\scriptstyle{tu}$} (-\x, \x-3) node[left,xshift=-.1cm] {$\scriptstyle{su}$};		
			\draw (-7,0) node {$\gamma' =$} (0,-3) node[below, xshift=0.2cm] {$u^{-n}$};
			\draw[<->] (0,-.1) -- (0,-2.9) node[pos=0.2,right] {\tiny{length$=n$}};
			
			\end{tikzpicture}
		\end{center}
		
	\end{exam}	
	
	{Many interesting questions remain regarding the relationship between both Deligne complexes and Cayley graphs of Artin monoids and their associated Artin groups.  Some of these questions are discussed in the introduction to this paper.}


\bibliographystyle{alpha}
\bibliography{MonoidsBib.bib}		

\newcommand{\etalchar}[1]{$^{#1}$}
\begin{thebibliography}{CCC{\etalchar{+}}97}

\bibitem[Alt98]{Altobelli1998}
Joseph~A. Altobelli.
\newblock {The word problem for Artin groups of FC type}.
\newblock {\em Journal of Pure and Applied Algebra}, 129(1):1--22, 1998.

\bibitem[BH11]{BridsonHaefliger2011}
Martin Bridson and Andre Haeflinger.
\newblock {\em {Metric Spaces Of Non-Positive Curvature}}.
\newblock Springer-Verlag, 2011.

\bibitem[Boyd20]{Boyd2020}
Rachael Boyd.
\newblock Homological stability for {A}rtin monoids.
\newblock {\em Proceedings of the London Mathematical Society},
  121(3):537--583, 2020.

\bibitem[{Bri}71]{Brieskorn1971}
E.~{Brieskorn}.
\newblock {Die Fundamentalgruppe des Raumes der regul{\"a}ren Orbits einer
  endlichen komplexen Spiegelungsgruppe.}
\newblock {\em Inventiones Mathematicae}, 12:57, 1971.

\bibitem[BS72]{BrieskornSaito1972}
E.~Brieskorn and K.~Saito.
\newblock Artin-{G}ruppen und {C}oxeter-{G}ruppen.
\newblock {\em Inventiones Mathematicae}, 17:245--271, 1972.

\bibitem[CCC{\etalchar{+}}97]{BrieskornSaitoTranslation}
C.~Coleman, R.~Corran, J.~Crisp, D.~Easdown, R.~Howlett, D.~Jackson, and
  A.~Ram.
\newblock Artin groups and coxeter groups.
\newblock Translation of Brieskorn, E. and Saito, K. \textit{Artin-{G}ruppen
  und {C}oxeter-{G}ruppen}, 1997.

\bibitem[CD93]{CharneyDavis93}
Ruth Charney and Michael Davis.
\newblock Singular metrics of nonpositive curvature on branched covers of
  {R}iemannian manifolds.
\newblock {\em Amer. J. Math.}, 115(5):929--1009, 1993.

\bibitem[CD95a]{CharneyDavis1995}
Ruth Charney and Michael Davis.
\newblock {The K($\pi$,1)-problem for hyperplane complements associated to
  infinite reflection groups}.
\newblock {\em Journal of the American Mathematical Society}, 8(3):597--627,
  1995.

\bibitem[CD95b]{CharneyDavis1995a}
Ruth Charney and Michael Davis.
\newblock {Finite K($\pi$,1)'s for Artin Groups}.
\newblock {\em Prospects in Topology, ed. by F. Quinn, Annals of Math Study
  138}, 1995.

\bibitem[Cha95]{Charney1995a}
Ruth Charney.
\newblock {Geodesic automation and growth functions for Artin groups of finite
  type}.
\newblock {\em Mathematische Annalen}, 301(1):307--324, 1995.

\bibitem[Cha99]{Charney1999}
Ruth Charney.
\newblock {Injectivity of the Positive Monoid for Some Infinite Type Artin
  Groups}.
\newblock {\em Geometric group theory down under, Proceedings of a Special Year
  in Geometric Group Theory}, page 103:111, 1999.

\bibitem[Cha00]{Charney2000}
Ruth Charney.
\newblock The {T}its conjecture for locally reducible {A}rtin groups.
\newblock {\em Internat. J. Algebra Comput.}, 10(6):783--797, 2000.

\bibitem[Cha07]{Charney2007}
R.~Charney.
\newblock An introduction to right-angled {A}rtin groups.
\newblock {\em Geometriae Dedicata}, 125:141--158, 2007.

\bibitem[CMW19]{CharneyMorrisWright2019}
Ruth Charney and Rose Morris-Wright.
\newblock {Artin groups of infinite type: trivial centers and acylindical
  hyperbolicity}.
\newblock {\em Proceedings of the American Mathematical Society}, 2019.

\bibitem[Cox33]{Coxeter1933}
H.~S.~M. Coxeter.
\newblock The complete enumeration of finite groups of the form
  {$R_i^2=(R_iR_j)^{k_{ij}}=1$}.
\newblock {\em Journal of the London Mathematical Society}, s1-10(1):21--25,
  1933.

\bibitem[Dav08]{Davis2008}
M.~W. Davis.
\newblock {\em The geometry and topology of {C}oxeter groups}, volume~32 of
  {\em London Mathematical Society Monographs Series}.
\newblock Princeton University Press, Princeton, NJ, 2008.

\bibitem[Del72]{Deligne1972}
P.~Deligne.
\newblock Les immeubles des groupes de tresses g\'en\'eralis\'es.
\newblock {\em Inventiones Mathematicae}, 17:273--302, 1972.

\bibitem[Dob06]{Dobrinskaya2006}
N.~{\`E}. Dobrinskaya.
\newblock Configuration spaces of labeled particles and finite
  {E}ilenberg-{M}aclane complexes.
\newblock {\em Proceedings of the Steklov Institute of Mathematics},
  252(1):30--46, Jan 2006.

\bibitem[DP99]{Dehornoy1999}
Patrick Dehornoy and Luis Paris.
\newblock {Gaussian Groups and Garside Groups, Two Generalisations of Artin
  Groups}.
\newblock {\em Proceedings of the London Mathematical Society}, 79(3):569--604,
  1999.

\bibitem[Gar69]{Garside1969}
F.~A. Garside.
\newblock {The braid group and other groups}.
\newblock {\em The Quarterly Journal of Mathematics}, 20(1):235--254, 01 1969.

\bibitem[God07]{Godelle2007}
Eddy Godelle.
\newblock {Artin-Tits groups with CAT (0) Deligne complex}.
\newblock {\em Journal of Pure and Applied Algebra}, 208(1):39--52, 2007.

\bibitem[GP12]{GodelleParis2012a}
Eddy Godelle and Luis Paris.
\newblock {Basic questions on Artin-Tits groups}.
\newblock {\em Configuration Spaces}, pages 299--311, 2012.

\bibitem[Hen85]{Hendriks1985}
Harrie Hendriks.
\newblock Hyperplane complements of large type.
\newblock {\em Invent. Math.}, 79(2):375--381, 1985.

\bibitem[HW08]{Haglund2008}
Fr{\'{e}}d{\'{e}}ric Haglund and Daniel~T. Wise.
\newblock {Special cube complexes}.
\newblock 17(5):1551--1620, 2008.

\bibitem[McD79]{McDuff1979}
Dusa McDuff.
\newblock On the classifying spaces of discrete monoids.
\newblock {\em Topology}, 18(4):313--320, 1979.

\bibitem[Mic99]{Michel1999}
J.~Michel.
\newblock A note on words in braid monoids.
\newblock {\em Journal of Algebra}, 215(1):366--377, 1999.

\bibitem[Ozo17]{Ozornova2017}
V.~Ozornova.
\newblock Discrete morse theory and a reformulation of the k($\pi$,
  1)-conjecture.
\newblock {\em Communications in Algebra}, 45(4):1760--1784, 2017.

\bibitem[Pao17]{Paolini2017}
Giovanni Paolini.
\newblock On the classifying space of {A}rtin monoids.
\newblock {\em Comm. Algebra}, 45(11):4740--4757, 2017.

\bibitem[Par02]{Paris2002}
Luis Paris.
\newblock {Artin monoids inject in their groups}.
\newblock {\em Commentarii Mathematici Helvetici}, 77(3):609--637, 2002.

\bibitem[Par14]{Paris2014}
Luis Paris.
\newblock {K($\pi$,1) conjecture for Artin groups}.
\newblock {\em Annales de la facult{\'{e}} des sciences de Toulouse
  Math{\'{e}}matiques}, 23(2):361--415, 2014.

\bibitem[PS19]{PaoliniSalvetti2018}
Giovanni Paolini and Mario Salvetti.
\newblock Proof of the {$K(\pi,1)$} conjecture for affine {A}rtin groups.
\newblock {\tt arXiv:1907.11795}, 2019.

\bibitem[Sag12]{Sageev2012}
Michah Sageev.
\newblock {CAT(0) cube complexes and groups}.
\newblock {\em Geometric group theory}, (0), 2012.

\bibitem[Sal94]{Salvetti1994}
Mario Salvetti.
\newblock The homotopy type of artin groups.
\newblock {\em Mathematical Research Letters}, 1:565--577, 01 1994.

\bibitem[vdL83]{VanderLek1983}
Harm van~der Lek.
\newblock {The homotopy type of complex hyperplane complements}.
\newblock {\em PhD thesis, Nijmegen}, 1983.

\bibitem[Vin71]{Vinberg1971}
\`E.~B. Vinberg.
\newblock Discrete linear groups that are generated by reflections.
\newblock {\em Izv. Akad. Nauk SSSR Ser. Mat.}, 35:1072--1112, 1971.

\end{thebibliography}

\end{document}